\documentclass[12pt]{amsart}
\usepackage[colorlinks=true, pdfstartview=FitV, linkcolor=blue, citecolor=blue, urlcolor=blue]{hyperref}

\usepackage{amssymb,amsmath, amscd, array}
\usepackage{times, verbatim}
\usepackage{graphicx}
\usepackage[english]{babel}
 \usepackage[usenames, dvipsnames]{color}
\usepackage{amsmath,amssymb,amsfonts}
\usepackage{enumerate}
\usepackage{anysize}
\marginsize{3cm}{3cm}{3cm}{3cm}
\input xy
\xyoption{all}
\usepackage{pb-diagram}
\usepackage[all]{xy}
\input xy
\xyoption{all}

\DeclareFontFamily{OT1}{rsfs}{}
\DeclareFontShape{OT1}{rsfs}{n}{it}{<-> rsfs10}{}
\DeclareMathAlphabet{\mathscr}{OT1}{rsfs}{n}{it}

\begin{document}
\theoremstyle{plain}

\newtheorem{theorem}{Theorem}[section]
\newtheorem{thm}[equation]{Theorem}
\newtheorem{prop}[equation]{Proposition}
\newtheorem{corollary}[equation]{Corollary}
\newtheorem{conj}[equation]{Conjecture}
\newtheorem{lemma}[equation]{Lemma}
\newtheorem{definition}[equation]{Definition}
\newtheorem{question}[equation]{Question}

\theoremstyle{definition}
\newtheorem{conjecture}[theorem]{Conjecture}
\newtheorem{example}[equation]{Example}
\numberwithin{equation}{section}

\newtheorem{remark}[equation]{Remark}

\newcommand{\Hecke}{\mathcal{H}}
\newcommand{\Liea}{\mathfrak{a}}
\newcommand{\Cmg}{C_{\mathrm{mg}}}
\newcommand{\Cinftyumg}{C^{\infty}_{\mathrm{umg}}}
\newcommand{\Cfd}{C_{\mathrm{fd}}}
\newcommand{\Cinftyfd}{C^{\infty}_{\mathrm{ufd}}}
\newcommand{\sspace}{\Gamma \backslash G}
\newcommand{\PP}{\mathcal{P}}
\newcommand{\bfP}{\mathbf{P}}
\newcommand{\bfQ}{\mathbf{Q}}
\newcommand{\Siegel}{\mathfrak{S}}
\newcommand{\g}{\mathfrak{g}}
\newcommand{\A}{\mathbb{A}}
\newcommand{\la}{\mathbb{\langle}}
\newcommand{\ra}{\mathbb{\rangle}}
\def\G{{\rm G}}
\def\B{{\rm B}}
\def\T{{\rm T}}
\def\SL{{\rm SL}}
\def\PSL{{\rm PSL}}
\def\GSp{{\rm GSp}}
\def\der{{\rm der}}

\newcommand{\wT}{\widehat{\T}}
\newcommand{\wG}{\widehat{\G}}
\newcommand{\wB}{\widehat{\B}}
\newcommand{\wrho}{\widehat{\rho}}
\newcommand{\Q}{\mathbb{Q}}
\newcommand{\Gm}{\mathbb{G}_m}
\newcommand{\Nm}{\mathbb{N}m}
\newcommand{\ii}{\mathfrak{i}}
\newcommand{\II}{\mathfrak{I}}

\newcommand{\kk}{\mathfrak{k}}
\newcommand{\nn}{\mathfrak{n}}
\newcommand{\tF}{\tilde{F}}
\newcommand{\p}{\mathfrak{p}}
\newcommand{\m}{\mathfrak{m}}
\newcommand{\bb}{\mathfrak{b}}
\newcommand{\Ad}{{\rm Ad}\,}
\newcommand{\ttt}{\mathfrak{t}}
\newcommand{\frakt}{\mathfrak{t}}
\newcommand{\U}{\mathcal{U}}
\newcommand{\Z}{\mathbb{Z}}
\newcommand{\bfG}{\mathbf{G}}
\newcommand{\bfT}{\mathbf{T}}
\newcommand{\R}{\mathbb{R}}
\newcommand{\ST}{\mathbb{S}}
\newcommand{\h}{\mathfrak{h}}
\newcommand{\bC}{\mathbb{C}}
\newcommand{\C}{\mathbb{C}}
\newcommand{\F}{\mathbb{F}}
\newcommand{\N}{\mathbb{N}}
\newcommand{\qH}{\mathbb {H}}
\newcommand{\temp}{{\rm temp}}
\newcommand{\Hom}{{\rm Hom}}
\newcommand{\Aut}{{\rm Aut}}
\newcommand{\rk}{{\rm rk}}
\newcommand{\Ext}{{\rm Ext}}
\newcommand{\End}{{\rm End}\,}
\newcommand{\Ind}{{\rm Ind}}
\newcommand{\ind}{{\rm ind}}
\newcommand{\Irr}{{\rm Irr}}
\def\circG{{\,^\circ G}}
\def\M{{\rm M}}
\def\H{{\rm H}}
\def\SS{{\rm S}}
\def\ZZ{{\rm Z}}
\def\diag{{\rm diag}}
\def\Ad{{\rm Ad}}
\def\As{{\rm As}}
\def\wG{{\widehat \G}}

\def\PGSp{{\rm PGSp}}
\def\Sp{{\rm Sp}}
\def\St{{\rm St}}
\def\GU{{\rm GU}}
\def\SU{{\rm SU}}
\def\U{{\rm U}}
\def\GO{{\rm GO}}
\def\GL{{\rm GL}}
\def\PGL{{\rm PGL}}
\def\GSO{{\rm GSO}}
\def\GSpin{{\rm GSpin}}
\def\GSp{{\rm GSp}}

\def\Gal{{\rm Gal}}
\def\SO{{\rm SO}}
\def\O{{\rm  O}}
\def\Sym{{\rm Sym}}
\def\sym{{\rm sym}}
\def\St{{\rm St}}
\def\Sp{{\rm Sp}}
\def\tr{{\rm tr\,}}
\def\ad{{\rm ad\, }}
\def\Ad{{\rm Ad\, }}
\def\rank{{\rm rank\,}}

\subjclass{Primary 11F70; Secondary 22E55}

\title{Character theory at a torsion element}
\author{Santosh Nadimpalli, Santosha Pattanayak, Dipendra Prasad}
\date{\today}

\thanks{
  DP thanks  Science and Engineering research board of the
  Department of Science and Technology, India for its support
through the JC Bose
National Fellowship of the Govt. of India, project number JBR/2020/000006.
SN is supported by the MATRICS grant MTR/2023/000306 from SERB, Department of Science and Technology, India.}

\address{Indian Institute of Technology, Kanpur, 208016}
\email{nvrnsantosh@gmail.com}
\address{Indian Institute of Technology, Kanpur, 208016}\email{santoshcu2@gmail.com}
\address{D.P.: Indian Institute of Technology Bombay, Powai, Mumbai-400076} 

\email{prasad.dipendra@gmail.com}
\maketitle

{\hfill \today}
\begin{abstract}{The paper begins by studying  the restriction of an irreducible 
finite dimensional representation of a complex 
reductive algebraic group to its principal ${\rm SL}_2(\mathbb C)$. 
This allows one to calculate the character of an irreducible representation of $\G$ at any element of its principal ${\rm SL}_2(\mathbb C)$, in particular,
at {\it principal elements} of a maximal torus which operates on all simple
roots by the same scalar, giving us another
proof of a theorem of Kostant on  the character values at the Coxeter conjugacy class,
\cite{Kos76}, and more generally for any power of the Coxeter element. 

We prove that these principal elements of order
$m$ in the adjoint group 
have smallest dimensional centralizer among elements of order $m$ in the adjoint group.
Our main theorem on character values becomes sharper 
when these principal elements of order
$m$ in the adjoint group are the only elements up to conjugacy having the
smallest dimensional centralizer among elements of order $m$ in the adjoint group. This
turns out to be true for most groups if $m|h$, where $h$ is the Coxeter 
number of $G$, though for certain pairs $(\G,m)$ it fails.

We define a  group $\G(m)$ which is the dual group of the (connected component of identity of the) centralizer
of the principal element of order $m$ in the adjoint group of $\wG$, which plays an important role in this work.
Our main theorem on character values depends on
identifying a particular irreducible representation of the simply connected cover of $\G(m)^{\der}$,  the derived
subgroup of $\G(m)$,
actually only its dimension, of highest weight $\rho/m-\rho_m$
(restricted to the maximal torus
of $\G(m)^{\der}$ where it is integral),
involving half-sum of positive roots of $\G$ and $\G(m)$,
which in turn
depends on 
the
precise heights of the simple roots in the
centralizer
of the powers of the Coxeter element. 
A preliminary analysis due to us is completed by
Patrick Polo in  \cite{Polo}.
  }
  \end{abstract}

\tableofcontents

\part*{\centering
  {Part A: Restriction to principal \texorpdfstring{${\rm SL}_2$}{}}}
\section{Introduction}

Let $\G$ be a group and $\theta$ an element of finite order in $\G$.
A well-known principle which permeates representation theory, though one which cannot be made precise except in
some very specific cases,
relates character values, $\chi_\pi(\theta)$, of irreducible representations $\pi$ of $\G$ at
$\theta$ with the character theory of the centralizer $\G^\theta$ of $\theta$ in $\G$.
It is closely related to the ``orbit method'', or the ``co-adjoint orbit method''.
One instance of this relationship is already there in one form of the classical
Schur orthogonality relations according to which for finite groups,
\[ \sum_{\pi} |\chi_\pi(\theta)|^2 = |\G^\theta|,\] 
the sum taken over the irreducible representations of $\G$, which closely resembles,
\[ \sum_{\pi'} |\dim \pi'|^2 = |\G^\theta|,\]
now the sum taken over the irreducible representations $\pi'$ of $G'= \G^\theta$, suggestive of the
possibility that when $\chi_\pi(\theta)$ is nonzero, it is related to the dimension
of an irreducible representation $\pi'$ of $G'=\G^\theta$.

A character relationship of this kind for the compact group $\O_{2n}(\R)$ at the conjugacy class
of $\theta \in \O_{2n}(\R)$ which is a reflection through any hyperplane, is the content of the
paper \cite{KLP}. In the context of Weyl groups, say $S_{n}$, and where $\theta$ is any
involution without fixed points, is the content of the paper \cite{LP}.

In this paper, we study the character theory of a compact connected Lie group
$\G$, equivalently, that of a reductive algebraic group, at  specific torsion elements
of $\G$, and relate it to the dimension of a representation of a group $G'$ which is exactly as in
\cite{KLP}, not exactly $\G^\theta$ but the dual group of a centralizer in the dual group.

The torsion elements of $\G$ at which we study the character theory of $\G$ are the
elements of $\G$, often called the {\it principal elements}, which are captured
by the principal ${\rm SL}_2(\mathbb C)$ in $\G$. Fixing a pair $(\B(\C),\T(\C))$
consisting of a Borel subgroup, and a maximal torus in $\G$, principal elements of $\G$
are those elements of $\T$ which act on all the simple roots of $\T$ on $\B$ by the same scalar.

A crucial property of these principal elements, operating on all simple roots say by $e^{2\pi i/d}$
is that these are the most ``symmetrical'' elements of $\G$ in the sense that their centralizer
is of the smallest possible dimension among elements  of order $\leq d$ in the adjoint group of $\G$.
We prove this as a consequence
of character theory of $\G$ restricted to the principal $\SL_2(\C)$ in Proposition \ref{minimal},
by comparing the orders of zeros of the numerator and the denominator of the character formula restricted
to $\SL_2(\C)$ around these principal elements.
This still leaves the possibility that there might be other conjugacy classes
in $\G$ containing an element of order $\leq d$ in the adjoint group which might achieve the
minimal dimensional centralizer.

However, if $d$ divides the Coxeter number $h$ of $\G$, then up to a small number of exceptions, it is the
only conjugacy class in an adjoint group
represented by an element of order $d$ with  the smallest dimensional
centralizer. This property sharpens  our understanding of the character theory
at principal elements, and all of Part B of this work is devoted to its analysis which is done in a
case-by-case way. Although we have not tried doing it here, it would be nice to analyze the cases $d < h$ when the principal elements
of order $d$ in the adjoint group are the unique conjugacy classes (in the adjoint group)
achieving the minimal dimensional centralizer which for the case of $\GL_n(\C)$ happens
exactly when $n \equiv -1,0,1 \bmod d$.

As an example, among involutions $\theta$ in $\G$, there is exactly one conjugacy
class of involutions such that $\G^\theta$ has the minimal possible dimension,
which is this principal element of $\G$ of order 2 in the adjoint group, and which
for $\GL_n(\C)$ is the diagonal element
 $(1,-1,1,-1,\cdots, (-1)^{n-1})$.

This work is inspired by the papers \cite{Pr1}, \cite{AK},
which calculate
character values
at special elements for $\GL_{mn}(\C), \Sp_{2mn}(\C),$
  $\SO_{2mn+1}(\C), \SO_{2mn}(\C)$ which for $\GL_{mn}(\C)$ are of the form:
\[ \underline{t}\otimes c_m = (\underline{t},  \underline{t}\cdot \omega, \cdots, \underline{t}\cdot
  \omega^{m-1}),\]  
  where $ \underline{t}=(t_1,\cdots, t_n)$ is a diagonal matrix in $\GL_n(\C) $, $\omega = e^{2\pi i/m}$, and
  $c_m$ is the diagonal matrix in $\GL_m(\C)$ given by $c_m=(1,\omega,\omega^2,\cdots, \omega^{m-1})$.
  The papers \cite{Pr1}, \cite{AK}  obtain a product formula for characters of classical groups at the element
  $\underline{t}\otimes c_m$ in terms of characters of smaller classical groups.
  A particular case of their work,  for $ \underline{t}=(t_1,\cdots, t_n) = (1,\cdots, 1)$, thus at the element of $\GL_{mn}(\C)$ (and similarly for classical groups)
  \[(1,1,\cdots, 1, \omega,\omega, \cdots, \omega,\cdots, \omega^{m-1}\cdots, \omega^{m-1}),\]
  where $\omega = e^{2\pi i/m}$, and each repetition is $n$ times,
  is a consequence of our work here, giving character value at such elements as dimension of
  certain representations of certain classical groups which, as noted in \cite{AK}, are not necessarily subgroups
  of the group involved.

  It may be added that \cite{Pr1}, \cite{AK}, as well as the works \cite{KLP}, and \cite{LP}
  relate not only the character of representations of $\G$ at $\theta$ with dimension of representations of a group $\G'$
  related to $\G^\theta$, but relate characters at {\it all}  elements of $\G'$ with character at
  {\it certain} elements of $\G$, an
  issue which needs serious considerations, but on which we will have nothing to say in this paper.

A recent work of C. Karmakar, cf. \cite{CK}, calculates the character theory of all
classical groups $\G$, and $\G_2$, at {\it all} elements $\theta$ of order 2, and finds that though there is a relationship to the centralizer group $\G^\theta$, it is more tenuous than here unless $\theta$ is a principal element
of $\G$.
  
Here is a brief summary of this work. After setting up some notations and recalling some preliminaries
in Section \ref{two}, Section \ref{three}  proves a factorization theorem about characters of $\G$ restricted to the
principal $\SL_2(\C)$ in which we bring out the role of the dual group $\wG$ associated to $\G$. Section \ref{three} contains a
proof of Proposition \ref{minimal} regarding the principal elements of $\G$ of order $m$ in the adjoint group
having the smallest dimensional centralizer among all elements of $\G$ of order $\leq m$ in the adjoint group.

Section \ref{four} contains the main theorem of this work which is Theorem \ref{conj} which calculates
the character of an irreducible representation $\pi_\lambda$ of $\G$ at principal elements of $\G$ of order $m$ in the adjoint group,
up to a sign, and a constant $d_m$, 
as the dimension of an irreducible representation 
(explicitly defined in terms of the highest weight
$\lambda$)
of the simply connected cover of  $\G(m)^{\der}$, the derived subgroup of $\G(m)$,
where $\G(m)$ is the dual group of the centralizer
of the principal element of order $m$ in the adjoint group of $\wG$.
The constant $d_m$ is also the dimension
of an irreducible fundamental representation of the simply connected cover of $\G(m)^{\der}$
which is furthermore minuscule. We leave the exact value
of the sign to the statement of Theorem \ref{conj}.
This theorem gives another
proof of a theorem of Kostant on  the character values at the Coxeter conjugacy class,
\cite{Kos76}, and  more generally for any power of the Coxeter element.

Section \ref{constant}, analyses the constant $d_m$ for classical groups, and
proves that in several cases the constant $d_m =1$.

Section \ref{examples} summarizes the calculation of the group $\G(m)$ from Part B of this work, and Section \ref{G_2}
uses Theorem \ref{conj} to calculate the character value at the unique conjugacy class of order 2 in the
exceptional group $\G_2$.

Section \ref{general} poses a question, due to the third author, of certain situations when
$\Theta_\lambda(x_0)$ must be zero, where $x_0$ is an arbitrary torsion element of $\G$, 
not necessarily a principal element.

Section \ref{adjoint} calculates the character of the restriction of the adjoint representation
of $\G$ to its principal $\SL_2(\C)$.
Recall that there is a famous theorem of Kostant which expresses the adjoint representation of $\G$ as a sum
of $\ell$ irreducible representations of the principal $\SL_2(\C)$ where $\ell = $ rank of $[\G,\G]$.
On the other hand, Theorem \ref{conj} expresses the character of an irreducible representation
of $\G$ restricted to the principal $\SL_2(\C)$ as a product over the set of positive roots of $\G$.
Neither of these two theorems gives a hint that, in fact, the character of the adjoint representation of $\G$
restricted to the principal $\SL_2(\C)$ has the simple factorization given in Theorem \ref{Adjoint}
which has at most 3 factors of the form $(1\pm z^d)$! Perhaps there is another point of view for this theorem! We have also
not managed to find a uniform statement for all groups for Theorem \ref{Adjoint} which is
proved by a case-by-case analysis.

Section \ref{converse} deals with the question whether an irreducible representation
of $\G$ is determined (say up to automorphisms) by its restriction to the principal
$\SL_2(\C)$, in particular for the group $\SL_n(\C)$, where this question amounts to
a question in combinatorics: whether a finite set $X$ of integers is determined (up to translation and reflection)
by the multiset,
\[ X-X = \{ a-b| a,b \in X\}.\]
There does not seem to be a simple answer to this question! This question has more structure
than another question often asked about reductive groups $\G$: how does one identify
irreducible representations of $\G$ with the same dimension? Here, instead, we are asking to identify
irreducible representations of $\G$ with the same restriction to the principal $\SL_2(\C)$. (Actually, for several simple groups $\G$,  the restriction
of an irreducible representation of $\G$ to the principal $\SL_2(\C)$
determines the irreducible representation, see for instance, Proposition \ref{unique},
so this strategy does not help
to identify irreducible representations of these groups having the same dimension.)

The aim of Part B is to analyze when the centralizer of an element $x_d \in \G$ of order $d$
in the adjoint group has minimal possible dimension, hoping that the principal
element $C_d$ of order $d$ in the adjoint group is the unique one with this property. For $d|h$, the Coxeter
number, the only exceptions are in some cases of $D_n,E_6,E_7$ which we explicitly analyze by hand for
classical groups one-by-one, and using Kac coordinates for exceptional groups again one-by-one. This part
starts with Section \ref{Kac} which reviews the Kac coordinates following Serre's exposition in \cite{serre_kac}.

\section{Notation and Preliminaries} \label{two}

Let $\G$ be a reductive algebraic  group over $\C$ with a fixed pair $(\T,\B)$ consisting of a maximal torus $\T$
and a Borel subgroup $\B$ with $\T\subset \B\subset \G$.
Let $X^*(\T)$ be the character group of $\T$, and $X_*(\T)$ be the co-character group of $\T$, with the natural
perfect pairing, denoted as $\langle -,- \rangle$,  of free abelian groups:
\[ X_*(\T) \times  X^*(\T) \rightarrow \Z.\]

Let $\Phi$ be the set of roots of $\T$  in $\G$, $\Phi^+$ the set of positive roots of $\T$  in $\B$,
and $\{\alpha_1,\alpha_2, \cdots, \alpha_n\}$ be the set of positive simple roots of $\T$ inside $\B$.

For each root $\alpha \in  X^*(\T)$, let  $\alpha^\vee  \in  X_*(\T)$ be the corresponding coroot with
\[ \langle \alpha^\vee, \alpha \rangle = 2.\]

Associated to the triple $(\G,\B,\T)$, let  $(\wG,\wB,\wT)$ be the dual group. The dual group
comes equipped with an identification,
\begin{eqnarray*}
  X^*(\T) & \longrightarrow &   X_*(\wT) \\
  \lambda  & \longrightarrow &  \hat{\lambda}.
\end{eqnarray*}

For each root $\alpha \in  X^*(\T)$, and a coroot  $\beta^\vee  \in  X_*(\T)$, there is the 
corresponding coroot  $\hat{\alpha} \in  X_*(\wT)$, and the root  $\hat{\beta} ^\vee  = \widehat{\beta^\vee} \in  X^*(\wT)$,
with,
\[ \langle \beta^\vee , \alpha \rangle  = \langle \hat{\alpha} , \hat{\beta}^\vee \rangle   \in \Z.\]

More generally as $\alpha \rightarrow \hat{\alpha}$
defined for roots $\alpha$ of $\T$ in $\G$, extends to an isomorphism between
$X^*(\T)$ and $X_*(\wT)$ (unlike $\alpha \rightarrow \alpha^\vee$), 
for any weight $\lambda \in X^*(\T)$, with co-weight $\hat{\lambda} \in X_*(\wT)$, and a coroot $\beta^\vee \in X_*(\T)$, we have,

\[ \langle \beta^\vee , \lambda \rangle  = \langle \hat{\lambda} , \hat{\beta}^\vee \rangle   \in \Z.\]

Let $\omega_i \in  X^*(\T)\otimes \Q$ be the fundamental weights for the triple $(\G,\B,\T)$, defined by:

\[ \la \alpha_j^\vee, \omega_i \ra = \delta_{ij},\]

and let  $\omega^\vee_i$ be the fundamental co-weights for the triple $(\G,\B,\T)$, defined by:

\[ \la \omega_i^\vee, \alpha_j \ra = \delta_{ij}.\]

Let $\rho \in  X^*(\T)\otimes \Q $ be half-sum of positive roots  for the triple $(\G,\B,\T)$. One knows that,
\[ \rho = \sum \omega_i.\]

Similarly, let $\rho^\vee \in  X_*(\T)\otimes \Q$ be half-sum of positive coroots, for which we have:
\[ \rho^\vee = \sum \omega^\vee_i.\]

Thus, we will be using,

\vspace{2mm}

\begin{tabular}{l c l}
  $\rho$ & $\in$ &   $X^*(\T)\otimes \Q,$ \\
  $\rho^\vee$ & $\in$ &   $X_*(\T)\otimes \Q,$ \\
  $\hat{\rho}$ & $\in$ &  $X_*(\wT)\otimes \Q,$ \\
  $\hat{\rho}^\vee$ & $\in$ &  $X^*(\wT)\otimes \Q$.
\end{tabular}
\vspace{2mm}

In this, using the
identification of $X^*(\T)$ with $X_*(\wT)$, and  that of $X_*(\T)$ with $X^*(\wT)$, we may
not differentiate between $\rho$ and $\hat{\rho}$, and  between $\rho^\vee$ and  $\hat{\rho}^\vee$. More generally,
we may not differentiate between a  weight $\lambda \in X^*(\T)$ with the corresponding
co-weight $\hat{\lambda} \in X_*(\wT)$.

For any root $\alpha \in \Phi^+$, define the height of $\alpha$, denoted $h(\alpha)$, by $h(\alpha) = \sum a_i$,
if $\alpha = \sum a_i \alpha_i$. As, $\la \omega_i^\vee, \alpha_j \ra = \delta_{ij}$, and as
$\rho^\vee = \sum \omega^\vee_i,$ we have:
\[ h(\alpha)  = \la  {\rho}^\vee, \alpha \ra.\]

Using the dual root system, we also have the notion $h(\alpha^\vee)$, but note that we may not have
$h(\alpha^\vee) =  h(\alpha) $.

For a dominant integral character $\lambda \in X^*(\T)$, we let
 $\pi_\lambda$ denote the irreducible finite dimensional representation of $\G(\C)$ with highest weight  $\lambda = \sum k_i \omega_i \in  X^*(\T)$
with $k_i \geq 0$, $k_i \in \Q$, where $\omega_i$ are the fundamental weights.

For a connected reductive group $\G$, let $\G^{\der}$ be the derived subgroup of $\G$, and let $\ZZ^0(\G)$
be the connected component of identity of the center of $\G$. Then $\G = \ZZ^0(\G) \cdot \G^{\der}$ with
$\ZZ^0(\G) \cap \G^{\der}$ a finite abelian group. Any irreducible representation of $\G$ restricts to an irreducible representation of
$\G^{\der}$. Conversely, any irreducible representation of $\G^{\der}$ extends to an irreducible representation
of $\G$ (not uniquely unless $\ZZ^0(\G)$ is trivial). For questions being studied in this
paper about character theory at elements in $\G^{\der}$, we can without loss of generality, assume that $\G$
is semisimple, and actually to the case of $\G$ a connected simple group (by which we mean that $\G$ has no connected
normal subgroup).

For understanding the representation theory of a semisimple group, or their characters, it suffices to assume
that $\G$ is simply connected which is preferable to assume for certain analysis as centralizer
of semisimple elements are then connected reductive groups. However, for our work, we mostly consider  centralisers
in $\wG$, and thus $\G$ being an adjoint group is better. On the other hand, it is not good to restrict to adjoint
groups $\G$, as then we are analyzing only certain representations of a general $\G$ (those with trivial
central character). For most of our work, we will deal with {\it all} simple groups.
Note also that the question of the dimension of the centraliser of elements is the same in an isogeny class
of groups, thus for the analysis done in Part B of this paper, it does not matter
which group in the isogeny class we choose to analyse, and thus either we deal with classical groups or with
the adjoint exceptional groups.

We will use $h = h(\G)$ to denote the Coxeter number of $\G$ (to be used only when $\G$ is a simple group).
Recall that if $\G$ is simple with a given pair $(\T \subset \B)$ consisting on a maximal torus $\T$ contained in a Borel subgroup $\B$ in $\G$, then the Coxeter
conjugacy class, denoted by $C_h$,
is represented by any element of $T$ which operates on all simple roots
of $\T$ in $\B$ by $e^{2 \pi i/h}$. All such elements of $\T$ differ from each other by an element of the center of $\G$, and are all in one conjugacy class of $\G$.
The order of any element in $C_h$ is either $h$ or $2h$
(and is $h$ if $\G$ is adjoint). For $m|h$, let $C_m = C_h^{h/m}$ which
is represented by any element of $T$ which operates on all simple roots
of $\T$ in $\B$ by $e^{2 \pi i/m}$.

\section{Restriction to the Principal ${\rm SL}_2(\mathbb C)$: the product formula} \label{three}
Let $\G$ be a connected complex reductive group with Lie
algebra $\mathfrak{g}$, and $W$ its Weyl group. 
Let $\T$ be a maximal torus and let 
$\B$ be a Borel subgroup of $\G$ with $\T\subset \B$. 
Let $\Phi$ be the set of roots of $\G$ with respect to 
$\T$, and let $\Phi^+$ be the set of positive 
roots of $\Phi$ contained in $\B$.
We denote by $\Delta$ the set of simple roots in $\Phi^+$. 
Let $X^\ast(\T)$ and  $X_\ast(\T)$
be the root lattice and co-root lattice of $\T$ respectively.
Let $\rho^\vee$ be half-sum of positive co-roots. It is to be noted
that  $\rho^\vee$ may not be a co-character of $T$, but  $2\rho^\vee$ is always.
However,  $\rho^\vee(z) =2\rho^\vee(\sqrt{z}) $, $z \in \C^\times$, is a meaningful element
of $\T(\C)$ up to the central element $2\rho^\vee(-1)$ in $\G$.

There exists a homomorphism 
\begin{equation}\label{principal}
 \psi = \psi_{\G}: {\rm SL}_2(\mathbb C)\rightarrow \G  
\end{equation}
which takes  a regular unipotent in $\SL_2(\C)$ to a regular unipotent element
in $\G$. Such a homomorphism is unique up to conjugacy, and 
 is called the
 principal $\SL_2(\C)$ in $\G$. We assume  that $\psi: {\rm SL}_2(\mathbb C)\rightarrow \G  $ takes the diagonal torus in $\SL_2(\C)$ to $\T$, and denote
 the corresponding map from $\Gm$ to $\T$, or to $\G$, also by $\psi$.

 The following property of  $\psi$ is essential to us:
\[ \psi = 2\rho^\vee:\mathbb{G}_m\rightarrow \T,\] 
equivalently, the action of $z \in \C^\times = \mathbb{G}_m(\C)$ on all the simple roots
of $\T$ in $\B$ is via $z \to z^2$. An element $t \in \T$ acting on all the simple roots by the same scalar is called a {\it principal element} for $(\G,\B,\T)$, and each such element (up to a central element) arises as $\psi(z)$ for $z \in \C^\times$.  The element in $\T$
\[ C_d = \psi(e^{\pi i/d}) = \rho^\vee(e^{2\pi i/d})=  (2\rho^\vee)(e^{2\pi i/2d}),\] which acts on all the simple roots by $\omega_d = e^{2\pi i/d}$ will be denoted by $C_d$. Observe that $C_d$ is well-defined up to conjugacy in $\G$, and that not all
principal elements of $\T$  for $(\G,\B,\T)$ which act on all simple roots by $e^{2\pi i/d}$ may belong to the conjugacy class
of $C_d$ because of the center of $\G$.

As $\rho^\vee$ is not necessarily a cocharacter, but
$2\rho^\vee$ is, the equality, $\rho^\vee(e^{2\pi i/d}) =   (2\rho^\vee)(e^{2\pi i/2d})$ should be taken as the definition of   $\rho^\vee(e^{2\pi i/d})$.

The following integrality statement will be useful to have for many of our assertions.

\begin{lemma}\label{integral}
  Let $\G$ be a connected reductive group, with a maximal torus $\T$ contained in a Borel subgroup $\B$ of $\G$,
  and the associated root datum. Let $\rho$ be half the sum of positive roots of $\T$ on $\B$, thus
  $\rho \in X^*(\T) \otimes \Q$.
  For any root $\alpha$ of $\T$ with coroot $\alpha^\vee$, we have the integrality statement,
  \[ \langle\rho, \alpha^\vee\rangle \in \Z.\]
  Similarly, if $\rho^\vee$ is half the sum of positive coroots, then,
  \[ \langle\rho^\vee , \alpha\rangle \in \Z.\]

\end{lemma}
\begin{proof}
  If $\G$ is simply connected , then the integrality of $\langle\rho, \alpha^\vee\rangle$ is obvious as in that case
  $\rho \in X^*(\T)$. As the pairing  $\langle\rho, \alpha^\vee\rangle$ is invariant under an isogeny,
  the proof follows for general reductive groups. The second part of the Lemma
 about the integrality of $ \langle\rho^\vee , \alpha\rangle$ follows from the first part using the dual group.
  \end{proof}

If $h$ is the Coxeter number of $\G$, $C_h$ will be called the Coxeter conjugacy class of $\G$.  Calculation of the characters of highest weight
irreducible representations of $\G$ at the elements $C_d$ is the principal aim of this paper. Towards this aim, we begin by calculating
characters at all elements of $\psi(\C^\times)$. The following theorem is available in one form or the other in many sources, and as we were finalizing this paper, we noticed the preprint \cite{Serre},
in which Theorem 2.8 has exactly this statement, and contains some historical remarks too. However, for the sake of completeness, we have decided to keep our proof.

\begin{theorem}\label{prod}
    Let $\G$ be a connected reductive group 
and let $\psi$ be a principal homomorphism as in 
\eqref{principal}. Let $\lambda\in X^\ast(\T)$ be 
a dominant weight and let $(\pi_\lambda, V(\lambda))$ be the 
irreducible representation of $\G$ with highest weight 
$\lambda$, $\Theta_\lambda$ its character, and $\Theta_\lambda(z)$ the character at the diagonal element
$(z,z^{-1})$ of  ${\rm SL}_2(\mathbb C)$ treated as an element of $\G$.  Then we have 
\begin{equation}\label{productformula}
    \Theta_\lambda(z)=
z^{-2\langle\lambda, \rho^\vee\rangle}
\dfrac{\prod_{\alpha\in \Phi^+}(1-z^{2\langle\lambda+\rho, \alpha^\vee\rangle})}{
  \prod_{\alpha\in \Phi^+}(1-z^{2\langle\rho, \alpha^\vee\rangle})},\end{equation}
where by Lemma \ref{integral}, the polynomials in $z$ inside the two products are polynomials in $z^2$, whereas
$z^{-2\langle\lambda, \rho^\vee\rangle}$ may be of odd degree.
\end{theorem}
\begin{proof}
    We are interested in the restriction of $(\pi_\lambda, V_\lambda)$
    to the group $\psi({\rm SL}_2(\mathbb{C}))$.
    Setting $t=\rho^\vee(z^2)$, for $z\in \mathbb{C}$, we write the Weyl numerator for the
    representation $\pi_\lambda$ at the element  $t=\rho^\vee(z^2)$, as
\begin{eqnarray*} \sum_{w\in W}(-1)^{l(w)}w(\lambda+\rho)(\rho^\vee(z^2)) & = &
\sum_{w\in W}(-1)^{l(w)}z^{2\langle 
  \lambda+\rho, w\rho^\vee\rangle}, \\
& = & \sum_{w\in W}(-1)^{l(w)}t_\lambda^{w\rho^\vee},  ~~~~~~~~~{\rm where} ~~~ t_\lambda=(\lambda+\rho)(z^2) \in \widehat{\T},\\
& = & t_\lambda^{\rho^\vee}
\prod_{\alpha\in \Phi^+}(1-t_\lambda^{-\alpha^\vee}), \\
& = &
z^{2\langle\lambda+\rho, \rho^\vee\rangle}
\prod_{\alpha\in \Phi^+}(1-z^{-2\langle \lambda+\rho, \alpha^\vee\rangle}),
\end{eqnarray*}
where the third equality is by the Weyl denominator formula
for the dual root system; note that in the LHS of the first equality,
$(\lambda+\rho) \in X^*(\T) \otimes \Q$, whereas in the RHS of the second equality,
$(\lambda+\rho) \in X_*(\wT) \otimes \Q$ which in the notation of the previous section
would have been written as $\hat{\lambda} + \hat{\rho}$, but we are not going to be
careful about displaying the difference between $\lambda \in X^*(\T)$ and
$\hat{\lambda} \in X_*(\wT)$. Note also that for all $t \in \wT$, $\alpha^\vee
\in X_*(\T) = X_*(\wT)$, $t^{\alpha^\vee} = \alpha^\vee(t) \in \C^\times.$

As the Weyl denominator is nothing but the Weyl numerator for $\lambda =0$, and since
$\Theta_\lambda(z)=\Theta_\lambda(z^{-1})$, we get that 
\[ \Theta_\lambda (z)
=z^{-2\langle\lambda, \rho^\vee\rangle}
\dfrac{\prod_{\alpha\in \Phi^+}(1-z^{2\langle \lambda+\rho, \alpha^\vee\rangle})}
{\prod_{\alpha\in \Phi^+}(1-z^{2\langle \rho, \alpha^\vee\rangle})},\]
proving the theorem.\end{proof}

\vspace{2mm}

\begin{remark}
    Recall that
  the Weyl character formula at an element $t \in \T$
  for the irreducible representation $\pi_\lambda$ of $\G$
  where $\lambda: \T \rightarrow \C^\times$ is a dominant weight for $(\G,\B,\T)$,
  is a quotient of the Weyl numerator by the Weyl denominator. The Weyl
  numerator depends on the pair $(t, \lambda)$, whereas Weyl denominator
  depends only on $t$, and is expressible as a product over positive roots
  of $\T$ in $\B$. In the proof above, the Weyl numerator for $(\G,\B,\T)$
  on the pair $(t, \lambda)$, for special choice of $t\in \T$ (which arise from the image of the principal
  $\SL_2(\C)$) becomes the  Weyl denominator for
  $(\wG,\wB,\wT)$ at special elements of $\wT$ which are constructed using $\lambda$. This explains the
  product formula. This also explains the presence of the dual group in this work.
   \end{remark}

We now draw some corollaries. The first and the most 
apparent is the Weyl dimension formula. For this, and for
other applications too, notice that
\[ \lim_{z\rightarrow 1} \frac{(1-z^n)}{(1-z)} = n.\]
This proves the Weyl dimension formula immediately.

As another corollary, note the following consequence of the product formula,
cf. Problem on page 520 of Serre's Collected works, volume III, which in the notes added
at the end of the Collected works on page 714, Serre says that it has been solved by V. Kac, cf.
Corollary 3.5 of \cite{Kac}.

\begin{corollary}
  Let $\G$ be any of the exceptional groups (which all have the property that
  for $h$ the Coxeter number, $h+1$ is a prime). Let $\pi$ be a non-minuscule
  fundamental representation of $\G$. Let $C_{h+1}$ be an element of a maximal
  torus $\T \subset \G$ acting by $e^{2 \pi i/(h+1)}$ on all the simple roots. Then
  the character of $\pi$ at $C_{h+1}$ is zero.
  \end{corollary}

\vspace{2mm}
As we evaluate the product formula, the following lemma will be useful.

\begin{lemma}\label{c-term}
  For an integer $m\geq 1$,  let $z = e^{2\pi i/m} + \epsilon$ with $|\epsilon|< 1$, a complex number.  
  For an integer $d$ which may be either positive or negative, let $f_d(z)=1-z^d$. Then $f_d(z)$ is
  an analytic function in the domain $|\epsilon|<1$  
with nonzero constant term $(1-e^{2d\pi i/m})$ if $m$ does not divide $ d$. If $m | d$,
$f_d(z)$ has zero constant term and with nonzero first  term   $-de^{-2\pi i/m} \epsilon$.
If $d \geq 0$, then  $f_d(z)$ is a polynomial in
$\epsilon$.
\end{lemma}
\vspace{5mm}

Since $\Theta_\lambda(z)$ for  $z = e^{2\pi i/m} + \epsilon$
is a polynomial in $\epsilon$,  comparing the order of zero in $\epsilon$ in  the numerator and the denominator 
of the polynomial \eqref{productformula},
we get the
following result. 
\vspace{5mm}

\begin{corollary} \label{inequality}
    Let $m$ be a positive integer and let 
    $\lambda$ be a dominant weight in $X^\ast(\T)$.
    Then, we have 
    \[\tag{2}   \# \{\alpha\in \Phi^+:
    m| \langle\rho, \alpha^\vee\rangle
   \} \leq 
     \# \{\alpha\in \Phi^+:
    m|\langle\lambda+\rho, \alpha^\vee\rangle
    \} .\]
    Further, $\Theta_\lambda(\rho^\vee(e^{2\pi i /m})) = \Theta_\lambda(e^{\pi i /m})\neq  0$ if and only if the above inequality is an equality.
\end{corollary}

\vspace{5mm}

Corollary \ref{inequality} will be interpreted as the dimension of
the centralizer of certain elements in the dual group as in the following proposition.
This result is due to J-P. Serre, see Theorem 1 of \cite{EKV}. His proof is similar but not exactly the same, as the proof given below, both relying on the restriction
of a representation of $\G$ to the principal $\SL_2(\C)$.

\vspace{2mm}

\begin{prop} \label{minimal}
  Let $m$ be a positive integer, and $\G$ a semisimple simply connected group.
  Then for the homomorphism $\hat{\rho}: \C^\times \rightarrow \wT \subset \wG$,
  the centralizer of  $\hat{\rho}(e^{2\pi i /m})$ in $\wG$ has minimal possible dimension among the
  centralizer of elements in $\wG$ of order dividing $m$.
\end{prop}

\begin{proof} As $\G$ is a semisimple simply connected  group, $\rho \in X^*(\T)$,
  and hence  $\hat{\rho}: \C^\times \rightarrow \wT \subset \wG$ is well-defined.
  
  According to a well-known theorem of Steinberg,
  the connected component of identity of the centralizer of an element $t$ in a maximal torus  $\wT \subset \wG$
    is generated by $\wT$ together with those root spaces of $\wT$ in $\wG$ for which the root takes
    the value 1 on $t$.
    Therefore, the centralizer of  $\hat{\rho}(e^{2\pi i /m})$ in $\wG$ is the dimension of $\wT$ plus the
    cardinality of the set of the left hand side of the inequality in
    equation (2) above. Similarly, the centralizer of
    $(\hat{\lambda}+ \hat{\rho})(e^{2\pi i /m})$ in $\wG$ is the dimension of $\wT$ plus the
    cardinality of the set of the right hand side of equation (2) above.

    Now, for any character $\lambda \in X^*(\T)$ treated as a cocharacter
    $\hat{\lambda} \in X_*(\wT)$, the adjoint action of $\hat{\lambda}(z) \in \wT$ on the $\alpha^\vee$ root space of $\wG$ is via $z \to
    z^{\langle \hat{\lambda}, \alpha^\vee \rangle} = z^{\langle \lambda, \alpha^\vee \rangle}$.
      Therefore,
  $\hat{\rho}(e^{2\pi i /m}) \in \wT$ acts trivially on the   $\alpha^\vee$ root space
of $\wT$ in $\wG$ 
    if and only if
    \[ m | \la \rho, \alpha^\vee\ra .\]

Similarly, 
$(\hat{\lambda} + \hat{\rho})  (e^{2\pi i /m}) \in \wT$
acts trivially on the   $\alpha^\vee$ root space
of $\wT$ in $\wG$ 
    if and only if
    \[ m | \la \lambda+ \rho, \alpha^\vee\ra .\]
    
    Thus Corollary \ref{inequality} proves that for any dominant integral
    weight $\lambda$, the dimension of the centralizer of
    $(\hat{\lambda} + \hat{\rho})  (e^{2\pi i /m})$  in $\wG$ is larger
    than or equal to the dimension of the centralizer of
    $ \hat{\rho}  (e^{2\pi i /m})$  in $\wG$.

Finally, we note that any element in $\wT$ of order $m$
is of the form  $\mu(e^{2\pi i /m})$, for some $\mu: \C^\times \rightarrow \wT $,
which can be assumed to be  of the form
$\mu = (\hat{\lambda}+\hat{\rho})$ for some $\lambda$ dominant integral.
To prove this, observe that
  $\mu(e^{2\pi i /m}) = (\mu +m \mu')
(e^{2\pi i /m})$ for any   $\mu, \mu': \C^\times \rightarrow \wT$,  and therefore assuming $\mu'$ is sufficiently dominant,  $\lambda= \mu +m \mu' -\hat{\rho}$ will be dominant, with  $\mu(e^{2\pi i /m}) = (\lambda+\hat{\rho})(e^{2\pi i /m})$,
completing the proof of the Proposition.
\end{proof}

\vspace{2mm}

{\bf Definition :} For a reductive group $\G$, and an integer $m\geq 1$, let $\G(m)$ be the reductive group
whose dual group is the connected component of identity of the
centralizer of  $\hat{\rho}(e^{2\pi i /m})= (2\hat{\rho})(e^{\pi i /m})
$ in $\wG$.  

\vspace{2mm}

The group $\G(m)$ is not a subgroup of $\G$, rather its dual group is contained in the dual group of $\G$.

\begin{remark} \label{themin}
  In Part B of the paper, we will prove that for $m|h$, the Coxeter number,
  the centralizer of  $\hat{\rho}(e^{2\pi i /m})$ in $\wG$ has {\it not only} the minimal possible dimension among the
  centralizer of elements in $\wG$ of order dividing $m$, it is the unique conjugacy class in $\wG$
  containing an element of order $\leq m$ of the minimal possible dimension
  for all classical groups except $D_{n}, E_6,E_7$. 
  It is true for $m=2$ and $m=h$ in all cases, the latter being a well-known result
  of Kostant, whereas
  for $m=2$, we are talking about centralizers of involutions which are well documented in the literature on
  symmetric spaces, such as in the book of Helgason, cf. \cite{Hel78}, and which allows one to check by hand that the centralizer
  of  $\hat{\rho}({-1})$ in $\wG$ has dimension which is strictly smaller than the dimension of the
  centralizer of other involutions; we will give independent proof of this in Part B of this paper.
  \end{remark}

\section{Character values on  principal elements} \label{four}
Let $\G_\lambda(m)$ be the reductive group 
with the root system
  \[\Phi_{\lambda,m}=  \{\alpha\in \Phi: m|\langle\lambda+\rho, \alpha^\vee\rangle\},\]
 and with with coroots $\alpha^\vee$ when
  $\alpha \in \Phi_{\lambda,m}$. By Lemma \ref{integral}, $ \langle\lambda+\rho, \alpha^\vee\rangle$ are integral.
  (To define the root datum of a group such as
   $\G_\lambda(m)$, we need to define
  $(X,R,X^\vee,R^\vee)$ which uses the same $X,X^\vee$ as for $\G$, and a certain subset of
  the set of roots and coroots of $\G$. Note that the set of roots
  $\{\alpha\in \Phi: m|\langle\lambda+\rho, \alpha^\vee\rangle \}$ is not a
  closed set of roots,
  because of which  $\G_\lambda(m)$ is not a subgroup of $\G$ (but is the dual of a subgroup of the dual group)). Further, $\Phi_m = \Phi_{\lambda,m}$
  for $\lambda = 0$ defines the group $\G(m)$ already introduced;  $\Phi^+_{\lambda,m}$, and  $\Phi^+_m$ are
  the corresponding positive roots.

  We begin with the following lemma that we will need.

  \begin{lemma}\label{sign}
    \[{\prod_{\alpha\in w \Phi^+}(1-z^{2\langle \rho, \alpha^\vee\rangle})} = (-1)^w
    z^{2\langle w^{-1}\rho - \rho, \rho^\vee \rangle }
      {\prod_{\alpha\in  \Phi^+}(1-z^{2\langle \rho, \alpha^\vee\rangle})}.\]
    \end{lemma}
  \begin{proof}
    The Lemma follows from the standard observation that
    \[ f_w(z)= {\prod_{\alpha\in w \Phi^+}(z^{-\langle \rho, \alpha^\vee\rangle}}-z^{\langle \rho, \alpha^\vee\rangle}),\]
      is skew-symmetric in $w$, i.e., $f_w(z) = (-1)^w f_1(z).$
    \end{proof}

  \begin{corollary} \label{add}
    Suppose $w(\lambda+\rho) = \rho +m \mu,$ then as a function of $\epsilon$,
    where $z = e^{\pi i/m} +\epsilon$, the leading term of 
${ z^{-2\langle\lambda, \rho^\vee\rangle}}
      /{\prod_{\alpha\in \Phi^+}(1-z^{2\langle \rho, \alpha^\vee\rangle})}$ equals the leading term of 
    $(-1)^w {\mu(c(\G))}/
      {\prod_{\alpha\in w \Phi^+}(1-z^{2\langle \rho, \alpha^\vee\rangle})}$, where
      $c(\G)$ is the central element $ 2\rho^\vee(-1)$ of $\G$ with $c(\G)^2=1$.   
    \end{corollary}
  \begin{proof}
    Under the condition  $w(\lambda+\rho) = \rho +m \mu,$ we find that
    $ w^{-1}\rho - \rho = \lambda - mw^{-1}(\mu)$, and therefore the identity
    in Lemma \ref{sign} becomes
    \[{\prod_{\alpha\in w \Phi^+}(1-z^{2\langle \rho, \alpha^\vee\rangle})} = (-1)^w
    z^{2\langle \lambda - mw^{-1}(\mu) , \rho^\vee \rangle }
      {\prod_{\alpha\in  \Phi^+}(1-z^{2\langle \rho, \alpha^\vee\rangle})},\]
      which can be rewritten as:

\[\dfrac{ z^{-2\langle\lambda, \rho^\vee\rangle}}
      {\prod_{\alpha\in \Phi^+}(1-z^{2\langle \rho, \alpha^\vee\rangle})} =
    (-1)^w  \dfrac{ z^{\langle m w^{-1}(\mu), 2\rho^\vee\rangle}}
      {\prod_{\alpha\in w \Phi^+}(1-z^{2\langle \rho, \alpha^\vee\rangle})}. \]
      
      Now, evaluating the two sides of this equality at  $z = e^{\pi i/m} +\epsilon$, we find that the leading term of the LHS is the same as the leading term
      of  $1/{\prod_{\alpha\in w \Phi^+}(1-z^{2\langle \rho, \alpha^\vee\rangle})},$
      multiplied by the leading term of
      $ (-1)^w  z^{-\langle m w^{-1}(\mu), 2\rho^\vee\rangle}$
        which as $(e^{\pi i/m})^m=-1$, is $(-1)^w w^{-1}(\mu)(  2\rho^\vee(-1))$. Now,
        $ 2\rho^\vee(-1)$ is a central element $c(\G)$ of $\G$ of order 1 or 2
        (which determines whether a selfdual representation of $\G$ is orthogonal or symplectic) $w^{-1}(\mu)(  2\rho^\vee(-1)) =\mu( c(\G)),$ completing the proof of the corollary.  \end{proof}

  In the following theorem, we will be using
  dominant integral weight $\lambda$ of $\T$ such that 
  treating $\lambda+\rho$ and $\rho$ as co-characters of $\wT$,  the elements 
    $(\lambda+\rho)(e^{2\pi i /m})$ and $\rho(e^{2\pi i /m})$ of $\wG(\C)$ are 
  conjugate in $\widehat{\G}(\mathbb{C})$, thus we must have:
  \[ w(\lambda+\rho) = \rho +m \mu,\]
  for some $w$ in the Weyl group of $\G$, and some cocharacter
  $\mu \in  X_*(\wT)$. Here we have used an elementary observation that
  a cocharacter is of the form $m \mu$  for some cocharacter
  $\mu \in  X_*(\wT)$ if and only if it is trivial on $e^{2\pi i /m}$.
  It may be noted that the  relationship $w(\lambda+\rho) = \rho +m \mu$
  does not uniquely determine the element $w$ of the Weyl group, nor the  cocharacter
  $\mu \in  X_*(\wT)$. However, we can further demand that $w$,
  which under the assumption $w(\lambda+\rho) = \rho +m \mu$, 
  has the property that   \[ w(\Phi_{\lambda,m}) = \Phi_{m},
  \]
  in fact takes the
  positive roots of $\Phi_{\lambda,m}$ to the positive roots of $\Phi_{m}$.
  Since any two positive systems of roots are conjugate by an element in the Weyl group, we can,
  and we will, assume that $ w(\Phi^+_{\lambda,m}) = \Phi^+_{m}$.
  The $(w,\mu)$ appearing in the theorem below
  are as here.

  \begin{theorem}\label{conj}
  Let $\G$ be a connected reductive group,
    and let $C_m$
    be the element $\psi(e^{\pi i/m})=   \rho^\vee(e^{2\pi i /m}) \in \G$. 
      Let $\pi_\lambda$ be the irreducible finite dimensional representation of $\G(\C)$ with highest weight  $\lambda \in  X^*(\T)$
      and with character $\Theta_\lambda: \T\rightarrow \C$.
      Then the dimension of the centralizer of $(\hat{\lambda}+\hat{\rho})(e^{2\pi i /m})$ in $\wG(\C)$ is greater than or equal
      to the
      dimension of the centralizer of $\hat{\rho}(e^{2\pi i /m})$, with equality of dimensions
      if and only if $\Theta_\lambda(C_m)\neq 0$.

    Assume that
$(\hat{\lambda}+\hat{\rho})(e^{2\pi i /m})$ and $\hat{\rho}(e^{2\pi i /m})$ are
    conjugate in $\widehat{\G}(\mathbb{C})$. Fix an element $w$ in the Weyl group
    of $\G$, and a character $\mu$ of the maximal torus $\T$ of $\G$, as above,  such that
     $w(\lambda+\rho) = \rho +m \mu$. Then,
    there is a constant $d_m$ which is the dimension of an irreducible
    (fundamental and minuscule) representation of the Lie algebra of $\G(m)^{\der}$, the
    derived subgroup of $\G(m)$, equivalently, that of the simply connected  cover of
 $\G(m)^{\der}$,
    such that
    
  \[\Theta_\lambda(C_m) =  (-1)^{w} \mu(c(\G)) \cdot d_\lambda/d_m,\]
where $c(\G)= 2\rho^\vee(-1) \in \T$, $d_\lambda$ is  the dimension of the
irreducible representation of the simply connected cover of $\G_\lambda(m)^{\der}$ of
dominant integral weight $(\lambda+\rho)/m-\rho_m$,
and where $\rho_m$ is half the sum of its positive roots. Further, $d_m=1$
if $G$ is of type $A_\ell,B_\ell$, and for $C_\ell$ if  $m$ is odd. Its complete description
is  due to Patrick Polo in \cite{Polo}.

Assume now that $m$ divides the Coxeter number $h$,
and that when $\G$ is of 
type $D_{n+1}$, $m$ is even and $2n/m$ is odd; 
when $G$ is of type $E_6$ (resp. $E_7$) $m\neq 4$ (resp. $m\neq 9$),
then $\Theta_\lambda(C_m)\neq 0$ if and only if
$(\hat{\lambda}+\hat{\rho})(e^{2\pi i /m})$ and $\hat{\rho}(e^{2\pi i /m})$ are
    conjugate in $\widehat{\G}(\mathbb{C})$.
\end{theorem}
\begin{proof}
  The first part of the Theorem about the necessary and sufficient condition
  for the non-vanishing of the character $\Theta_\lambda(C_m)$ is exactly the conclusion
  of Corollary \ref{inequality}.

  We therefore assume now that
  $w(\lambda+\rho) = \rho +m \mu$ with $w$ an element in the Weyl group of $\G$ with
$ w(\Phi^+_{\lambda,m}) = \Phi^+_{m}$.  

  By Corollary \ref{add},
  \begin{eqnarray*}
    \Theta_\lambda (z)
&  = &  z^{-2\langle\lambda, \rho^\vee\rangle}
\dfrac{\prod_{\alpha\in \Phi^+}(1-z^{2\langle \lambda+\rho, \alpha^\vee\rangle})}
      {\prod_{\alpha\in  \Phi^+}(1-z^{2\langle \rho, \alpha^\vee\rangle})}\\
       & = & (-1)^w \mu(c(G))\cdot
\dfrac{\prod_{\alpha\in \Phi^+}(1-z^{2\langle \lambda+\rho, \alpha^\vee\rangle})}
      {\prod_{\alpha\in w \Phi^+}(1-z^{2\langle \rho, \alpha^\vee\rangle})},\end{eqnarray*}
      where $z$ stands for the notation
  $\psi(z)$ introduced earlier which for the principal element $C_m$ corresponds
  to $z = e^{\pi i/m}$)

  It suffices therefore to calculate,
  \begin{eqnarray*} \Theta'_\lambda (z)
& = & 
\dfrac{\prod_{\alpha\in \Phi^+}(1-z^{2\langle \lambda+\rho, \alpha^\vee\rangle})}
      {\prod_{\alpha\in w \Phi^+}(1-z^{2\langle \rho, \alpha^\vee\rangle})},\\
&= &  
\dfrac{\prod_{\alpha\in \Phi^+_{\lambda,m}}(1-z^{2\langle \lambda+\rho, \alpha^\vee\rangle})}
      {\prod_{\alpha\in w \Phi^+_{\lambda,m}} (1-z^{2\langle \rho, \alpha^\vee\rangle})} \cdot
      \dfrac{\prod_{\alpha \in \Phi^+ \setminus \Phi^+_{\lambda,m}}(1-z^{2\langle \lambda+\rho, \alpha^\vee\rangle})}
            {\prod_{\alpha \in  w(\Phi^+ \setminus \Phi^+_{\lambda,m})} (1-z^{2\langle \rho, \alpha^\vee\rangle})}, 
    \end{eqnarray*}
  where,
  \[\Phi_{\lambda,m}=  \{\alpha\in \Phi: m|\langle\lambda+\rho, \alpha^\vee\rangle\},\]
  is the set of roots of the group $\G_\lambda(m)$. (We remind ourselves that
  by Lemma \ref{integral}, we have the integrality assertion $\langle\lambda+\rho, \alpha^\vee\rangle \in \Z$.)

  Let us write,
\begin{eqnarray*} \Theta'_\lambda (z)
            & = & B_\lambda(z) \cdot C_\lambda(z),      \end{eqnarray*}

where,

\begin{eqnarray*}
     B_\lambda(z) & = & 
      \dfrac{\prod_{\alpha  \in \Phi^+ \setminus \Phi^+_{\lambda,m}}(1-z^{2\langle \lambda+\rho, \alpha^\vee\rangle})}
            {\prod_{\alpha \in w(\Phi^+ \setminus
                \Phi^+_{\lambda, m})} (1-z^{2\langle \rho, \alpha^\vee\rangle})},\\
              C_\lambda(z) & = &  \dfrac{\prod_{\alpha\in \Phi^+_{\lambda,m}}(1-z^{2\langle \lambda+\rho, \alpha^\vee\rangle})}
      {\prod_{\alpha\in w \Phi^+_{\lambda, m}} (1-z^{2\langle \rho, \alpha^\vee\rangle})}.
\end{eqnarray*}

          By Lemma \ref{c-term},  each term in the
           product representing $B_\lambda$
($\alpha \in  \Phi^+ \setminus \Phi^+_{\lambda,m}$ in the numerator, and ${\alpha \in w( \Phi^+ \setminus \Phi^+_{m})}$ in the denominator)
          evaluated at $z = e^{\pi i/m} + \epsilon$
          is an analytic function
          in $\epsilon$ (in a neighborhood of zero) with an explicit nonzero constant term.
          (To apply Lemma \ref{c-term}, it is better to write $z = e^{\pi i/m} + \epsilon$ as
          $z = e^{2\pi i/(2m)} + \epsilon$, then the condition in Lemma \ref{c-term}
          is about $2m|2\langle \lambda+\rho, \alpha^\vee\rangle$,
          or $m|\langle \lambda+\rho, \alpha^\vee\rangle$, defining $\Phi_{\lambda,m}$.)

          Thus the  product for $B_\lambda$ 
evaluated at $z = e^{\pi i/m} + \epsilon$, gives rise to:

\begin{eqnarray*}
  B_{{\lambda}}(e^{\pi i/m} + \epsilon) & = &    \dfrac{\prod_{\alpha \in \Phi^+ \setminus
      \Phi^+_{\lambda,m}}(1-e^{2\pi i \langle \lambda+\rho, \alpha^\vee\rangle/m})}
  {\prod_{\alpha  \in w( \Phi^+ \setminus \Phi^+_{\lambda, m})} (1-e^{2\pi i \langle \rho, \alpha^\vee\rangle/m})}
  +\epsilon O(\epsilon),
  \\
   & =&      \dfrac{\prod_{\alpha \in \Phi^+ \setminus \Phi^+_{\lambda, m}}(1-e^{2\pi i \langle \lambda+\rho,  \alpha^\vee\rangle/m})}
   {\prod_{\alpha \in \Phi^+ \setminus \Phi^+_{\lambda, m}} (1-e^{2\pi i \langle \rho, w \alpha^\vee\rangle/m})} +\epsilon O(\epsilon), \\
   & =&      \dfrac{\prod_{\alpha \in \Phi^+ \setminus \Phi^+_{\lambda, m}}(1-e^{2\pi i \langle w(\lambda+\rho), w \alpha^\vee\rangle/m})}
   {\prod_{\alpha \in \Phi^+ \setminus \Phi^+_{\lambda, m}} (1-e^{2\pi i \langle \rho, w \alpha^\vee\rangle/m})} + \epsilon O(\epsilon),\\
&= & 1 + \epsilon O(\epsilon).
\end{eqnarray*}
where $\epsilon O(\epsilon)$  is an analytic function
in $\epsilon$ with zero constant term, and where
in the last equality, we have used
$ w(\lambda+\rho) = \rho + m \mu$.

Therefore,

\[ B_\lambda(e^{\pi i/m}) = 1.\]

Next we note that by Lemma \ref{c-term}, each term in the product expansion of $C_\lambda(z)$,
  evaluated at $z = e^{\pi i/m} + \epsilon$, both in the numerator and the denominator, are analytic functions
  in $\epsilon$ with zero constant term, and explicit nonzero first term, because of which the product expansion
  for $C_\lambda(z)$ 
 evaluated at $z = e^{\pi i/m} + \epsilon$, gives rise to:

    \[C_\lambda( e^{\pi i/m} + \epsilon)=
    \dfrac{\prod_{\alpha\in \Phi^+_{\lambda, m}}\langle\lambda+\rho, \alpha^\vee\rangle}
          {\prod_{\alpha\in w \Phi^+_{\lambda, m}}\langle\rho, \alpha^\vee\rangle} + \epsilon O'(\epsilon),\]
          where $\epsilon O'(\epsilon)$ is an analytic function
          in $\epsilon$ (in a neighborhood of zero) whose constant term is zero.

        Observe that:
            \[\prod_{\alpha\in \Phi^+_{\lambda, m}}\langle\lambda+\rho, \alpha^\vee\rangle = m^{|\Phi^+_{\lambda, m}|}\prod_{\alpha\in \Phi^+_{\lambda, m}}\langle \{(\lambda+\rho)/m-\rho^\lambda_m\} +\rho^\lambda_m, \alpha^\vee\rangle, \]
          where
          $\rho^\lambda_m $ is half the sum of positive roots in  $\Phi^+_{\lambda, m}$.

Therefore,

\begin{eqnarray*} C_\lambda(e^{\pi i/m}) & = &
  \dfrac {\prod_{\alpha\in \Phi^+_{\lambda, m}}
              \langle \{ (\lambda+\rho)/m-\rho^\lambda_m\} +\rho^\lambda_m, \alpha^\vee\rangle}{
            \prod_{\alpha\in \Phi^+_{\lambda, m}}\langle\rho^\lambda_m, \alpha^\vee\rangle} \cdot
          \dfrac{\prod_{\alpha\in \Phi^+_{\lambda, m}}\langle\rho^\lambda_m, \alpha^\vee\rangle} 
                {\prod_{\alpha\in w \Phi^+_{\lambda, m}}\langle\rho/m, \alpha^\vee\rangle}, \\
&=& \dfrac {\prod_{\alpha\in \Phi^+_{\lambda, m}}
              \langle \{ (\lambda+\rho)/m-\rho^\lambda_m\} +\rho^\lambda_m, \alpha^\vee\rangle}{
            \prod_{\alpha\in \Phi^+_{\lambda, m}}\langle\rho^\lambda_m, \alpha^\vee\rangle} \cdot
          \dfrac{\prod_{\alpha\in \Phi^+_{\lambda, m}}\langle\rho^\lambda_m, \alpha^\vee\rangle} 
                {\prod_{\alpha\in \Phi^+_{\lambda, m}}\langle\rho/m, w\alpha^\vee\rangle},
                \\ & = & d_\lambda / d_m,
          \end{eqnarray*}
where, using the fact that
$w(\Phi^+_{\lambda,m}) = \Phi^+_{m}$, we have $w\rho_m^\lambda= \rho_m$, and so:

\begin{eqnarray*}
  d_m & = &   \dfrac{\prod_{\alpha\in \Phi^+_{\lambda, m}}\langle\rho/m, w\alpha^\vee\rangle}{\prod_{\alpha\in \Phi^+_{\lambda, m}}\langle w \rho^\lambda_m,  w \alpha^\vee\rangle}  = \dfrac{\prod_{\alpha\in \Phi^+_{ m}}\langle \rho/m,
    \alpha^\vee\rangle}{\prod_{\alpha\in \Phi^+_{ m}}\langle\rho_m, \alpha^\vee\rangle}  =
  \dfrac{\prod_{\alpha\in \Phi^+_{ m}}\langle(\rho/m -\rho_m)+\rho_m,
    \alpha^\vee\rangle}{\prod_{\alpha\in \Phi^+_{ m}}\langle\rho_m, \alpha^\vee\rangle},  \\
  d_\lambda & =&  \dfrac {\prod_{\alpha\in \Phi^+_{\lambda, m}}
              \langle \{ (\lambda+\rho)/m-\rho^\lambda_m\} +\rho^\lambda_m, \alpha^\vee\rangle}{
    \prod_{\alpha\in \Phi^+_{\lambda, m}}\langle\rho^\lambda_m, \alpha^\vee\rangle}.
  \end{eqnarray*}

By the Weyl dimension formula, $d_m$ is the dimension of the highest weight irreducible
representation of the simply connected  cover of $\G(m)^{\der}$ of highest weight $(\rho/m -\rho_m)$ restricted to
the maximal torus of $\G(m)^{\der}$.
The calculation of the highest weight $(\rho/m -\rho_m)$
is partly executed in Section \ref{constant} for certain classical groups, and is completed
by Patrick Polo in \cite{Polo}.  In many cases, $(\rho/m -\rho_m)$ is
trivial on the derived subgroup $\G(m)^{\der}$
of $\G(m)$ hence $d_m=1$, and in all cases,
$(\rho/m -\rho_m)$ is a fundamental weight
of a minuscule representation of the simply connected cover of
$\G(m)^{\der}$, which is explicitly identified, and $d_m$ calculated in \cite{Polo}.

Finally,  $d_\lambda$ is also, by the Weyl dimension formula, the dimension of the 
highest weight module of the simply connected cover of $\G_\lambda(m)^{\der}$,
with highest weight,
the restriction of the weight
              \[\mu = (\lambda+\rho)/m-\rho^\lambda_m,\]
              to the maximal torus of $\G_\lambda(m)^{\der}$ which is the intersection
              of $\T$ with $\G_\lambda(m)^{\der}$; see Remark \ref{integrality} below.
              For this, we need to check that $\mu = (\lambda+\rho)/m-\rho^\lambda_m$
              is a dominant integral weight for
                          $\G_\lambda(m)^{\der}$, i.e.,

              \[ \la (\lambda+\rho)/m -\rho^\lambda_m, \alpha^\vee \ra ,\]
              are integral and $\geq 0$
              for all positive coroots $\alpha^\vee$ of $\G_\lambda(m)^{\der}$.

              By the definition
of $\G_\lambda(m)^{\der}$,  $\la \lambda+\rho, \alpha^\vee \ra  \geq m $
all positive coroots $\alpha^\vee$ of $\G_\lambda(m)^{\der}$. As $\rho^\lambda_m$ is  half
the sum of positive roots of $\Phi_{\lambda,m}$, $ \la \rho^\lambda_m, \alpha^\vee \ra  =1$ for all simple roots $\alpha$ of  $\Phi_{\lambda,m}$. Therefore,
\[ \la (\lambda+\rho)/m -\rho^\lambda_m, \alpha^\vee \ra  \geq 0 ,\]
for all simple roots of  $\Phi_{\lambda,m}$, and therefore for all
positive roots of $\Phi_{\lambda,m}$.

Further, by the definition of $\Phi_{\lambda,m}$,  $\la (\lambda+\rho), \alpha^\vee \ra  $
is a multiple of $m$ for all positive coroots $\alpha^\vee$ of $\G_\lambda(m)^{\der}$,
therefore
$(\lambda+\rho)/m$ is an integral weight in $\Phi_{\lambda,m}$, and hence
$(\lambda+\rho)/m -\rho_m^\lambda$ is also an integral weight in  $\Phi_{\lambda,m}$.

  Under  the restrictions imposed on the pair $(\G,m)$ in the last paragraph of the statement of the Theorem, it
  will be proved in Part B of this paper,  that if
    $(\lambda+\rho)(e^{2\pi i /m})$ and $\rho(e^{2\pi i /m})$ are
     not conjugate in $\widehat{\G}(\mathbb{C})$,
     then the dimension of the  centralizer
 of $(\lambda+\rho)(e^{2\pi i /m})$ in $\wG$
     is
  strictly larger than the dimension of the centralizer of $\rho (e^{2\pi i /m})$, and hence
  $\Theta_\lambda(C_m) =0$ by Corollary \ref{inequality}.
\end{proof}

\begin{remark} \label{integrality}
  In the above proof, we have used the following fact about reductive groups $\H$ with a maximal torus
  $\SS$, roots
  $\{\alpha\}$, and coroots $\{\alpha^\vee\}$. A dominant weight $\lambda \in X^*(\SS)\otimes \Q$ defines a
  representation of the simply connected cover of the derived group $\H^{\der}$ if and only if the following integrality condition is satisfied:
  \[ \langle \lambda, \alpha^\vee\rangle \in \Z ~~~~~~~~~~~~ \forall {~~\rm coroots~~} \alpha^\vee.\]
  \end{remark}

\begin{remark}
  The weight $\rho/m-\rho_m$, and more generally  $(\lambda+\rho)/m-\rho^\lambda_m$ are typically not
  integral weights for the group $\G(m)$ (resp., $\G_{\lambda}(m)$), and are integral only on the
  simply connected cover of the derived subgroup
  $\G(m)^{\der}$ (resp., $\G_{\lambda}(m)^{\der}$) by the previous remark, i.e.,
  these characters restricted to the maximal torus of  $\G(m)^{\der}$ (resp., $\G_{\lambda}(m)^{\der}$)
  are algebraic characters on the maximal tori in the simply connected cover of these derived groups.
  For instance if $m=h$, the Coxeter number of a simple group $\G$, then $\G(m)$
  is a torus, and $\rho_m$ is the trivial character, in which case  $\rho/m-\rho_m = \rho/m$. So, we will
  be asking if the character  $\rho/m$ of the maximal torus $\T=\G(m)$ makes sense as an
  algebraic character of $\T$, which is not true.
  \end{remark}

As a corollary to the above theorem, we note the following theorem of Kostant, \cite{Kos76},
given in this precise form in \cite{Pr2}.
\begin{corollary}
  Let $\G$ be a simple simply connected group over $\C$ with
  $h$ the Coxeter number, and $C_h$, 
    the Coxeter conjugacy class in $\G(\C)$.
        Let $\pi_\lambda$ be a finite dimensional representation of 
$\G(\C)$ with highest weight $\lambda: \T \rightarrow \C^\times$.
Then the character $\Theta_\lambda$ of $\pi_\lambda$
at $C_h$ is either $1,0,-1$, and it is nonzero if and only if 
$(\widehat{\lambda\cdot \rho}){}(e^{2\pi i /h})$ 
is conjugate to ${\wrho}(e^{2\pi i /h})$, 
both of them belonging to the Coxeter conjugacy class in $\wG$. 
If $\Theta_\lambda(C_h) \not = 0$, then the central character of $\pi_\lambda$ is trivial,
and hence $\pi_\lambda$ can be treated as a representation of an adjoint group,
with simply connected  dual group $\wG^{\rm sc}$. The conjugacy of
$(\widehat{\lambda\cdot \rho}){}(e^{2\pi i /h})$ and ${\wrho}(e^{2\pi i /h})$ in  $\wG^{\rm sc}$
is by a unique element of the Weyl group (because of the regularity of ${\wrho}(e^{2\pi i /h})$ in the simply connected group
$\wG^{\rm sc}$), call it $w_\lambda$. Then,
\[\Theta_\lambda(C_h) = (-1)^{\ell(w_\lambda)}.\]
\end{corollary}

\begin{proof}
  For the Coxeter number $h$, $\G(h)$ is a torus, and hence $\G(h)^{\der} = 1$
  and all its irreducible representations are
  one dimensional, therefore  $\Theta_\lambda(C_h) = 0,\pm 1$ is part of the conclusion of Theorem
  \ref{conj}, which further implies that
 $\Theta_\lambda(C_h)$ is 
  nonzero if and only if $(\widehat{\lambda\cdot \rho}){}(e^{2\pi i /h})$ 
  is conjugate to ${\wrho}(e^{2\pi i /h})$ in $\wG$. It is well-known and easy to prove that
  $C_h$ and $z\cdot C_h$, where $z$ belongs to the center of $\G$ are conjugate in $\G$. Therefore,
  if $\Theta_\lambda(C_h) \not = 0$, the central character of $\pi_\lambda$ must be trivial. Now, we use
  Theorem \ref{conj} for the adjoint group allowed by the previous remark.
  \end{proof}
\begin{example}
For $\G= \SL_2(\C)$, let $C_2 = \begin{pmatrix}
  i & 0 \\
   0 & -i  
\end{pmatrix}$ be the Coxeter conjugacy class in $\SL_2(\C)$.
Let $\pi_n = \Sym^n(\C^2)$ be the irreducible representation  
of $\SL_2(\C)$ of highest weight $n$ and dimension $(n+1)$. It is easy to see that
the character $\Theta_n(C_2)$ of $\pi_n$ at $C_2$ is nonzero if and only if $n$ is even,
in which case,
\[ \Theta_n(C_2) = (-1)^{n/2}.\]
In this case, 
$\lambda = n$, $\rho = 1$, therefore for
$w(\lambda+\rho) = \rho + 2 \mu$ to hold, we can have $w=1$, and $\mu=n/2$, or
$w=-1$, and $\mu = -(n+2)/2$. In this case $c(\SL_2(\C))=-1$. Therefore, if
$w=1$, and $\mu=n/2$, or $w=-1$, and $\mu = -(n+2)/2$, in either case,
Theorem \ref{conj}, gives the value $\Theta_n(C_2) = (-1)^{n/2}$ as expected.
\end{example}

\section{The constant $d_m$ for the group $\G_\lambda(m)$} \label{constant}
The aim of this section is to calculate the constant
\[
  d_m  =    \dfrac{\prod_{\alpha\in \Phi^+_{0, m}}\langle\rho/m, \alpha^\vee\rangle}{\prod_{\alpha\in \Phi^+_{0, m}}\langle\rho_m, \alpha^\vee\rangle}  =
\dfrac{\prod_{\alpha\in \Phi^+_{0, m}}\langle(\rho/m -\rho_m)+\rho_m, \alpha^\vee\rangle}{\prod_{\alpha\in \Phi^+_{0, m}}\langle\rho_m, \alpha^\vee\rangle},  \]
which appears in Theorem \ref{conj}, and is the dimension of an irreducible representation of the simply connected 
cover of $\G(m)^{\der}$ of highest weight $\rho/m -\rho_m$.

Observe that for any simple
root $\alpha$ of  $\Phi^+_{0, m}$,
\[ \langle\rho_m, \alpha^\vee\rangle = 1,\]
whereas \[ \langle\rho, \alpha^\vee\rangle = h(\alpha^\vee),\]
where $h(\alpha^\vee)$ denotes the height of $\alpha^\vee$ as a co-root of $\G$.

Therefore the constant $d_m=1$ if all simple roots of  $\Phi^+_{0, m}$
are of height $m$, an assertion we prove in the following proposition
in some cases. This analysis is completed in \cite{Polo}
to prove that for $\G$ a simple group,  at most one
simple root of  $\Phi^+_{0, m}$ is of height greater than $m$, in which case the
height of that extra simple root of  $\Phi^+_{0, m}$  is $2m$. Thus $d_m$ is the dimension of a fundamental representation
of the simply connected  cover of $\G(m)^{\der}$ which furthermore turns out to be minuscule.

\begin{prop} \label{classical}
  Let $\Phi$ be a simple root system of type $A_l$ ($l\geq 1$), $C_l$
  ($l\geq 2$) or $B_l$ ($l\geq 2$). Let $m\in \mathbb{N}$ with $m<h$.
  Assume that $m$ is odd when $\Phi$ is of type
  $B_l$. Let $\Phi^+$ be a choice of positive roots with $\Delta$ its
  simple roots.  Let $\Phi_m$ be the set
  $\{\alpha\in \Phi: m| {\rm ht}(\alpha)\}$.  The set
  $\Delta_m=\{\alpha\in \Phi^+: {\rm ht}(\alpha)=m\}$ is a basis for the
  root system $\Phi_m$.
     \end{prop}
    \begin{proof}
      When $\Phi$ is of type $A_l$, the elements of $\Phi$ correspond to
      connected subsets of $\Delta$ hence
      any connected subset of $\Delta$ of length a multiple of $m$ is the disjoint
      union of  connected subset of $\Delta$ of length equal to $m$, proving that
      roots of length $m$ for $A_l$ generate all positive roots of length a multiple
      of $m$, proving the proposition in this case.

      Next, let
      $\Phi$ be a root system of type $C_l$ and let $\widetilde{\Phi}$
      be a root system of type $A_{2l-1}$. 
      Let $V$ be an Euclidean space of dimension $2l$, 
      with orthonormal basis $$\{f_{-l},f_{-l+1},\dots, f_{-1}, f_{1}
      ,f_{2},\dots, f_{l}\}.$$ Let $\widetilde{\Phi}$
      be the set 
      $$\{f_i-f_j:-l\leq i\neq j\leq l\}.$$
      The set of vectors $\widetilde{\Phi}$ is a root 
      system of type $A_{2l-1}$.
      Let $\widetilde{\Delta}$ be the set 
      $$\{f_{-l}-f_{-l+1},f_{-l+1}-f_{-l+2},\dots, f_{-1}-f_{1},
      f_1-f_2,\dots f_{l-1}-f_l\},$$
      and $\widetilde{\Delta}$ is a basis for $\widetilde{\Phi}$. 
       Let $V'$ be the 
      quotient of $V$ by the space spanned by 
      $\{f_{-i}+f_{i}:1\leq i\leq l\}$. 
      Let $\Phi$ be the 
      image of $\widetilde{\Phi}$ under the 
      natural map $\phi:V\rightarrow W$.  
      Note that $\Phi$ is a root system
      of type $C_l$  and  $\Delta=\phi(\widetilde{\Delta})$ is 
      a basis for $\Phi$ and thus
     ${\rm ht}(\phi(\alpha))={\rm ht}(\alpha)$, for all $\alpha\in \widetilde{\Phi}$. 
       Hence, for a root system of type $C_l$, the proposition 
       follows from the $A_{2l-1}$ case.

       We come to the case where the root system is of
      type $B_l$ and $m$ is odd. Using induction on the height of
      roots $\alpha\in \Phi_m$, we show that
      $\alpha=\sum_{\beta\in \Delta_m}n_\beta\beta$, for some
      $n_\beta\in \mathbb{N}$. The roots of the form
$$\epsilon_i=\sum_{i\leq k\leq l}\alpha,\ (1\leq i\leq l)$$
and
$$\epsilon_i-\epsilon_j=\sum_{i\leq k<j}\alpha_k\ (1\leq i<j\leq l)$$
correspond to connected subsets of the set of simple roots, and the
proposition follows.  Consider the long root of the form
$$\gamma=\epsilon_i+\epsilon_j=
\sum_{i\leq k<j}\alpha_k+2\sum_{j\leq k\leq l}\alpha_k,\ (1\leq
i<j\leq l).$$
Since $m$ is odd, we may assume that $0<j-i<m$. 
 If $l-j+1>m-(j-i)$, then we can write
$\gamma=\gamma_1+\gamma_2$ where
$$\gamma_1=\sum_{i\leq k<i+m}\alpha_k$$
and 
$$\gamma_2=\sum_{j\leq k<i+m}\alpha_k+2\sum_{i+m\leq k\leq l}\alpha_k.$$
Note that $\gamma_1, \gamma_2\in \Phi$ and
${\rm ht}(\gamma_2)<{\rm ht}(\gamma)$.  In the case where
$l-j+1 \leq m-(j-i)$, we note that
$${\rm ht}(\gamma)=2(l-j+1)+j-i\leq 2m-(j-i)<2m.$$
Hence, we get that ${\rm ht}(\gamma)=m$. This proves the proposition.
     \end{proof}

     \begin{remark}
       From the above proposition, we deduce that the constant $d_m$ is
       equal to $1$ for groups $G$ of type $A_n$ and $B_n$ (we need to
       switch the roles of $\G$ and $\wG$ in Proposition \ref{classical}).  When
     $m$ is odd and $m|2n$, we have $d_m=1$ for groups of type
     $C_n$. 
     \end{remark}
     
     \section{Examples} \label{examples}

     In this section we summarize the results obtained in later sections
     on $\G(m)$ for $\G$ a classical group. Theorem \ref{conj} therefore proves
      that when dealing with
     character theory of a classical group at a principal element, not only are we
     dealing with character theory of $\GL_a(\C)$ for some integer $a$, and classical groups of the same type
     as $\G$, but also classical groups of different types than $\G$, an observation which
     comes up in the work of Ayyer-Kumari in \cite{AK}.

     \vspace{1cm}
     
     \begin{center}\label{table_cen}
      \begin{tabular}{| m{2cm} |  m{8cm} | m{4cm} |}
        \hline
        Group  & $\G(d)$ &  \\
        \hline
      & &   \\
      $\GL_{n}(\C)$&  $\GL_a(\C)^d$ & $a = n/d$ \\
     & &  \\
      \hline
     & & \\
      $\SO_{2n+1}(\C)$  & $\GL_a(\C)^{d/2}$ & $d$ even, $a = 2n/d$ \\
     & & \\
      
      &  $\GL_{2a}(\C)^{(d-1)/2} \times \SO_{2a+1}(\C)$  & $d$ odd, $a=2n/d$   \\
 & & \\
      \hline
  & & \\    
      $\Sp_{2n}(\C)$ &  $\GL_a(\C)^{(d-2)/2} \times \SO_a(\C)
      \times \Sp_a(\C) $ & $d$ even, $a= 2n/d$ even \\
      & & \\
      &  $\GL_a(\C)^{(d-2)/2} \times \SO_{a+1}(\C) \times \Sp_{a-1}(\C) $ & $d$ even, $a= 2n/d$ odd, \\
      & & \\
      &  $\GL_a(\C)^{(d-2)/2} \times  \Sp_{2a}(\C) $ & $d$ odd, $a=2n/d$,\\
    & &   \\
      \hline
     & &  \\
      
      $\SO_{2n+2}(\C)$ &  $\GL_{a}(\C)^{(d-1)/2} \times \SO_{a+2}(\C)$  & $d$ odd, $a= 2n/d$,\\
      & & \\

      &  $\GL_a(\C)^{(d-2)/2} \times \SO_{a}(\C) \times \SO_{a+2}(\C) $ & $d$ even, $a= 2n/d$ even, \\
     & &  \\
               & $\GL_a(\C)^{(d-2)/2} \times \SO_{a+1}(\C) \times \SO_{a+1}(\C) $ & $d$ even, $a= 2n/d$ odd.\\
        \hline
      \end{tabular}
  \end{center}

\section{Explicit calculation for $\G_2$} \label{G_2}

We specialize Theorem \ref{conj} for $m=2$ for $\G_2$ in which case 
Theorem \ref{conj} gives a precise recipe the character of irreducible representations of $\G_2(\C)$ at the unique
conjugacy class of order 2 represented by $C_2$.

We take $\{\alpha_1,\alpha_2\}$  as a basis of the root system of $\G_2$
for which  the positive roots and coroots  following  Bourbaki are given as follows: \\

\begin{itemize}

\item [1)] $\alpha_1$ with $\alpha_1^\vee = \alpha_1$,
  \item [2)]$\alpha_2$ with $\alpha_2^\vee = \alpha_2/3$,
    \item[3)] $\alpha_3=\alpha_1+\alpha_2$ with $\alpha_3^\vee = \alpha_3 = \alpha^\vee_1+3\alpha_2^\vee$,
    \item[4)] $\alpha_4=2\alpha_1+\alpha_2$ with $\alpha_4^\vee= \alpha_4 =2 \alpha^\vee_1+3\alpha_2^\vee $,
    \item[5)] $\alpha_5=3\alpha_1+\alpha_2$ with $\alpha_5^\vee = \alpha_5/3 = \alpha^\vee_1+\alpha_2^\vee$,
    \item[6)] $\alpha_6=3\alpha_1+2\alpha_2$ with $\alpha_6^\vee = \alpha_6/3 = \alpha^\vee_1+2\alpha^\vee_2 $.
\end{itemize}

Let $\omega_1,\omega_2$ be the fundamental weights for $\G_2$, with
\[\la \omega_i, \alpha^\vee_j \ra = \delta_{i,j},  {\rm ~~~ for ~~~} 1 \leq i,j \leq 2. \]

Therefore for $\lambda =  k\omega_1+\ell \omega_2$  highest weight of a representation of $\G_2$,
we have $\lambda+\rho = (k+1)\omega_1 + (l+1)\omega_2$, and we find that,

\begin{itemize}

\item [1)] $\la \lambda+\rho, \alpha_1^\vee \ra = (k+1)$,
  \item [2)] $\la \lambda+\rho, \alpha_2^\vee \ra = (l+1)$,
    \item[3)] $\la \lambda+\rho, \alpha_3^\vee \ra = (k+1) +3(l+1) = k+3l+4$,
    \item[4)] $\la \lambda+\rho, \alpha_4^\vee \ra = 2(k+1) + 3(l+1)= 2k+3l+5$,
    \item[5)] $\la \lambda+\rho, \alpha_5^\vee \ra = (k+1) + (l+1) = k+l+2$,
    \item[6)] $\la \lambda+\rho, \alpha_6^\vee \ra = (k+1) + 2(l+1) = k+2l+3$,
\end{itemize}

\vspace{4mm}

Using the above information, we have the following cases.

\vspace{4mm}

\begin{enumerate}

\item   \{$k$ and $\ell$ are both odd\}:  In this case,
  $\la \lambda+\rho, \alpha_i^\vee \ra $ is even for all $i=1,2, \cdots ,6$.

\item   \{$k$ and $\ell$ are both even\}:  In this case,
$\la \lambda+\rho, \alpha_i^\vee \ra $ is even for $i=3,5$, giving the product of
these even values to be:
\[ (k+3l+4)(k+l+2).\]

\item \{$k$ is even and $l$ is odd\}:  In this case,
$\la \lambda+\rho, \alpha_i^\vee \ra $ is even for $i=2,4$, giving the product of
these even values to be:
\[ (l+1)(2k+3l+5).\]

\item \{$k$ is odd and $l$ is even\}:  In this case,
$\la \lambda+\rho, \alpha_i^\vee \ra $ is even for $i=1,6$, giving the product of
these even values to be:
\[ (k+1)(k+2l+3).\]

\end{enumerate}

Now, we note that:
$\la \rho, \alpha_i^\vee \ra $ is even for $i=3,5$, giving the product of
these even values to be 8.  Thus our theorem on character values on powers of the Coxeter element 
proves the following result. This result is proved earlier by different methods
in \cite{Re1}, and in \cite{CK}.

\begin{prop}
  Let $V_{k,l}$ be the highest weight representation of $\G_2(\C)$ with highest weight $k\omega_1+l\omega_2$, where $\omega_1,\omega_2$
  are the fundamental representations of $\G_2(\C)$, with character $\Theta_{k,l}(C_2)$ at the unique conjugacy class $C_2$ of order $2$.
  Then we have 
\[ \Theta_{k,l}(C_2)= \begin{cases} 
  0 , \hspace{5.5cm}   \text{if} \  k,l \  are \ odd \\
  \\
   {\frac{(k+l+2) \times (3l+k+4)}{8}} , \hspace{3cm}   \text{if} \  k,l \  are \ even \\

  \\
  -{ \frac{(k+1) \times (k+2l+3)}{8},} \hspace{3.5cm}   \text{if} \  k \ odd,l  \ even \\

  \\
  
-{ \frac{(l+1) \times (3l+2k+5)}{8},} \hspace{3.5cm}  \text{if} \  k \ even,l  \ odd 
\end{cases} \]
  
\end{prop}

\section{Character at a general torsion element} \label{general}

Most of the paper is written around the character theory at torsion elements
captured by the principal $\SL_2(\C)$. In this section, we pose the
following question, attributing it to the third author.

\begin{question} Let $(\G,\B,\T)$ be a triple consisting of a semisimple simply connected group $\G$, a Borel subgroup
  $\B$, and a maximal torus $\T \subset \B$. 
  Let $\pi_\lambda$ be the highest weight module corresponding to a dominant weight $\lambda: \T\rightarrow \C^\times$, with character $\Theta_\lambda$. Let  $x_0$ be a  torsion element of $\G$, of  order $d$ in the adjoint group of $\G$.
  Treat $\lambda\cdot \rho$ as the cocharacter:
    $\lambda \cdot \rho:   \C^\times \rightarrow \wT$. Then if the 
 dimension of the centralizer of    $(\lambda \cdot\rho)(e^{2\pi i/d})$ in $\wG$ is greater than   the 
   dimension of the centralizer of   $x_0$  in $\G$, then is it true that:
   \[ \Theta_\lambda(x_0)=0?\]

   If the 
 dimension of the centralizer of    $(\lambda\cdot \rho)(e^{2\pi i/d})$ in $\wG$ is equal to the 
   dimension of the centralizer of   $x_0$  in $\G$, then,
   $ \Theta_\lambda(x_0)$ is up to a constant, the dimension of an irreducible representation
   of a reductive group whose dual group is the connected component of identity of
   the centralizer of    $(\lambda \cdot\rho)(e^{2\pi i/d})$ in $\wG$?
       \end{question}

\begin{remark}
  This question is inspired by  findings in \cite{CK} on character theory at 2 torsion elements for classical groups.
  Let us
  treat a special case of the question about zeros of characters at regular elements when in the
  Weyl character formula, the denominator is nonzero, so at such elements, the character value is zero
  if and only if the Weyl numerator is zero at $x_0$. One way to force Weyl numerator to be zero is to
  force $(\lambda \cdot \rho)(x_0)$ which appears in the Weyl numerator
  to be invariant under a simple reflection (though neither $\lambda \cdot \rho$ nor
  $x_0$ are left invariant by any non-trivial element of the Weyl group!).
  In this case, we are given that the centralizer of    $(\lambda \cdot\rho)(e^{2\pi i/d})$ in $\wG$ has larger
  dimension than the dimension of $\wT$, hence   $(\lambda \cdot\rho)(e^{2\pi i/d})$ is centralized by
  an involution in the Weyl group $s_\alpha$.

  We will treat here the case when $x_0$ is not only an element of order $d$ in the adjoint group of $\G$,
  but in $\G$ itself. The general case is as treated in  Lemma 1 of \cite{Pr2}.
  Now, note that
  $(\lambda \cdot \rho)$ can be treated as a character not only of $\T(\C)$, but also of its $d$-torsion
  points, $\T[d]$,
  and that the character group of $\T[d]$ can be identified
  to $\wT[d]$. As the element  $(\lambda \cdot \rho)(e^{2\pi i/d})$ in $\wG$
  is invariant under $s_\alpha$,
  the character
$(\lambda \cdot \rho)$ of $\T[d]$ is also invariant under $s_\alpha$, in particular,
  $(\lambda \cdot \rho)(x_0)$ which appears in the Weyl numerator
  is
  invariant under $s_\alpha$. Thus the Weyl numerator is both invariant and skew-invariant under
  $s_\alpha$, hence $\Theta_\lambda(x_0)=0$; see the proof
of Theorem 2 in \cite{Pr2} for a similar argument.
  \end{remark}

\section{Adjoint representation restricted to principal $\rm{SL}_2(\mathbb C)$} \label{adjoint}

Let ${\rm Ad}:\G \rightarrow {\rm Aut}(\mathfrak g)$ be the adjoint representation of $\G$.  In this section
we derive a formula for the character
$\Theta_\g(z)$
of the adjoint representation when restricted to the diagonal element $ (z,z^{-1})$ of the
principal $\rm{SL}_2(\mathbb C)$. Let $\Delta=\{\alpha_1, \alpha_2,\dots, \alpha_l\}$
be the set of simple roots contained in $\Phi^+$. Note that the adjoint representation $\mathfrak g$ of $\G$ is an irreducible representation of $\G$ if $\G$
is simple and has the dominant weight, the highest root $\alpha_0$ of $\Phi$.
We recall that the height of $\alpha=\sum k_i\alpha_i \in \Phi$ is defined to be $ht(\alpha) = \sum k_i$. For any two weights $\lambda$ and $\mu$, we say that $\lambda \geq \mu$ if $\lambda-\mu \in \mathbb Z_{\geq 0}\alpha_1+\cdots \mathbb +\mathbb{Z}_{\geq 0}\alpha_l$.

The following theorem calculates the character of the adjoint representation. It may be noticed that although
Theorem \ref{prod} expresses the character of a representation as a product over the  positive roots, in this
case the product simplifies considerably, and is a product of at most three terms! That the  expression simplifies so much is not obvious also
from the general theorem of Kostant on the restriction of the adjoint representation to the principal $\SL_2(\C)$ which expresses it
as a sum of $\ell$ irreducible representations of $\SL_2(\C)$, where $\ell$ is the rank of $\G$.

\vspace{5mm}

\begin{theorem} \label{Adjoint}
\end{theorem}
  \begin{tabular}{ c c l l }

   $\Theta_\g(z)$  & = & $\displaystyle z^{-2n}\frac{(1-z^{2n})(1-z^{2n+4})}{(1-z^2)^2},$ & {\rm ~~if} $\G=A_n$.  \\
  
   & = & $\displaystyle z^{-(4n-2)}\frac{(1-z^{4n})(1-z^{4n+2})}{(1-z^2)(1-z^4)},$ & {\rm ~~if} $\G=B_n$ or $C_n$.  \\  

   & = & $\displaystyle  z^{-(4n-6)}\frac{(1-z^{2n})(1+z^{2n-4})(1-z^{4n-2})}{(1-z^2)(1-z^4)},$ & {\rm ~~if} $\G=D_n$. \\

    & = & $ \displaystyle z^{-22}\frac{(1-z^{16})(1-z^{18})(1-z^{26})}{(1-z^2)(1-z^6)(1-z^8)},$ & {\rm ~~if} $\G=E_6$. \\

  & = & $\displaystyle  z^{-34}\frac{(1-z^{24})(1-z^{28})(1-z^{38})}{(1-z^2)(1-z^8)(1-z^{12})}$, & {\rm ~~if} $\G=E_7$. \\

   &= & $\displaystyle z^{-58}\frac{(1-z^{40})(1-z^{48})(1-z^{62})}{(1-z^2)(1-z^{12})(1-z^{20})}$, & {\rm ~~if} $\G=E_8$. \\

  &= & $\displaystyle  z^{-22}\frac{(1-z^{16})(1-z^{24})(1-z^{26})}{(1-z^2)(1-z^8)(1-z^{12})}$, & {\rm ~~if} $\G=F_4$. \\

   & = & $\displaystyle z^{-10}\frac{(1-z^{14})(1-z^{16})}{(1-z^2)(1-z^8)}$, & {\rm ~~if} $\G=\G_2$. 
\end{tabular}

  \begin{proof} The principal $\SL_2(\C)$ inside classical groups is
    very simple to define. For example, in the case of $\SL_{n+1}(\C)$, the principal
    $\SL_2(\C)$ inside  $\SL_{n+1}(\C)$ is the embedding through $\Sym^n(\C^2)$, and therefore the adjoint representation
     of $\SL_{n+1}(\C)$ restricted to   the principal
     $\SL_2(\C)$ is  \[\Sym^n(\C^2) \otimes  \Sym^n(\C^2)- 1,\]
     whose character at the diagonal element $(z,z^{-1})$ of $\SL_2(\C)$ is:
     \[ \frac{(z^{n+1}-z^{-(n+1)})^2}{(z-z^{-1})^2} - 1 =
     z^{-2n}\frac{(1-z^{2n})(1-z^{2n+4})}{(1-z^2)^2}.\]

     Similar analysis can be done in the case of $\Sp(2n,\C)$ and $\SO(n,\C)$ where the adjoint representation is respectively $\Sym^2(V)$, and $\Lambda^2(V)$, and $V$ itself in the case of $\Sp(V)$, $\dim V = 2n$,
     is a representation of $\SL_2(\C)$
     which is $\Sym^{2n-1}(\C^2)$, and in the case of
     $\SO(2n+1,\C), V = \Sym^{2n}(\C^2)$, whereas for the case of
     $\SO(2n,\C), V = \Sym^{2n-2}(\C^2)$.

     We will not give detailed proof based on Theorem \ref{prod} in the case of exceptional groups except to give a flavor of the arguments for the case of $E_8$. For $E_8$, one knows that the adjoint representation with highest weight
     $\alpha_0$ is the fundamental representation $\omega_8 = \alpha_0$. 

The product formula derived in Theorem \ref{prod} for a dominant weight $\lambda$ of $\G$ is the following: $$\Theta_{\lambda}(z)=z^{-2\langle\lambda, \rho^\vee\rangle}
\dfrac{\prod_{\alpha\in \Phi^+}(1-z^{2\langle \lambda+\rho, \alpha^\vee\rangle})}
{\prod_{\alpha\in \Phi^+}(1-z^{2\langle \rho, \alpha^\vee\rangle})},$$
which specializing to the adjoint representation for which
we get that 

$$\Theta_\g(z)=z^{-2\langle\alpha_0, \rho^\vee\rangle}
\dfrac{\prod_{\alpha\in \Phi^+}(1-z^{2\langle \alpha_0+\rho, \alpha^\vee\rangle})}
{\prod_{\alpha\in \Phi^+}(1-z^{2\langle \rho, \alpha^\vee\rangle})}$$

Note that $\langle \rho, \alpha^{\vee}\rangle=ht(\alpha^{\vee})$ for all $\alpha \in \Phi^+$. We also have $\langle \rho, \alpha_0^{\vee}\rangle=ht(\alpha_0)=h-1$, where $h$ is the Coxeter number of $\Phi$. So we get that

$$\Theta_\g(z)=z^{-2(h-1)}
\prod_{\alpha\in \Phi^+} \frac{(1-z^{2(ht(\alpha^{\vee})+\langle \alpha_0, \alpha^\vee\rangle)})}{(1-z^{2ht(\alpha^{\vee})})}.$$

If $\alpha \in \Phi^+$ such that $\langle \alpha_0, \alpha^\vee \rangle=0$, then the term $(1-z^{2\langle \rho, \alpha^\vee\rangle})$ in the numerator and the denominator cancels out. So we have 

\begin{equation}\label{char}
\Theta_\g(z) =z^{-2(h-1)}
\prod_{\{\alpha\in \Phi^+: \langle \alpha_0, \alpha^\vee\rangle \neq 0\}} \frac{(1-z^{2(ht(\alpha^\vee)+\langle \alpha_0, \alpha^\vee\rangle)})}{(1-z^{2ht(\alpha^\vee)})}.
\end{equation}

As the highest weight for $E_8$ is \[\alpha_0= 2 \alpha_1+3\alpha_2 + 4 \alpha_3 + 6 \alpha_4 + 5 \alpha_5 + 4 \alpha_6+ 3 \alpha_7 + 2 \alpha_8,\]
$\alpha_8$ appears in any positive root of $E_8$ with coefficient 0, 1, 2.
The set of positive $\alpha$'s with the coefficient of $\alpha_8$ being 0 is nothing but the set of positive roots of $E_7$ which is generated by the simple roots $\{\alpha_1,\alpha_2,\cdots, \alpha_7\}$ inside $E_8$. The group $E_7$ is the semi-simple part of a maximal Levi subgroup of $E_8$ of a parabolic $P=MN$ inside $E_8$ with $N$ a two-step nilpotent group in which
the root subgroup $\alpha_0$ is the center, the quotient by which is a 56 dimensional irreducible representation of $E_7$, for which all the roots have $\alpha_8$ coefficient equal to 1.

Since one knows the height function on the set of all positive roots of any reductive group (as the dual partition afforded by the exponents of the group, which in the case of $E_8$ is $\{1,7, 11,13, 17,19, 23, 29\}$), therefore also for
the roots in $N$ (being the difference of the height functions on $E_8$ and $E_7$). As all roots in $N$ except $\alpha_0$ have the coefficient of $\alpha_8$ equal to
1, the equation \ref{char} can be valuated, yielding the claimed result for $E_8$. \end{proof}

 \section{Restriction to principal $\SL_2(\C)$ determines the representation?} \label{converse}

 We begin with the following proposition -- which follows by a direct computation using Theorem \ref{prod}, whose proof will be omitted. 

 \begin{prop}\label{unique}
For any simple group $\G$ of rank $\leq 4$, an irreducible representation of
 $\G$ is uniquely determined by its restriction to the principal $\SL_2(\C)$ (up to an outer automorphism).
   \end{prop}

 However,  Proposition \ref{unique} fails in general,  the
 smallest group for which it fails being $\SL_6(\C)$. Thus there exists irreducible
 representations $\pi_\lambda$ and $\pi_\mu$, 
 with highest weights $\lambda$ and $\mu$ of  $\SL_6(\C)$ such that 
 $\pi_\lambda$ and $\pi_\mu$ have isomorphic restriction
 to a principal ${\rm SL}_2(\mathbb C)$ subgroup but, $\lambda$
 and $\mu$ are not related to each other by the outer automorphism of
 $\SL_6(\C)$.

 For instance, let 
 $$\lambda=3\varpi_3+\varpi_4+2\varpi_5$$
 and 
 $$\mu=4\varpi_2+\varpi_4+\varpi_5$$
 be dominant integral weights of ${\rm SL}_6(\mathbb C)$ written as a sum
 of fundamental weights. Using the 
 product formula from Theorem \ref{prod}, the 
 representations 
 $\pi_\lambda$ and $\pi_\mu$ have isomorphic restrictions
 to the principal ${\rm SL}_2(\mathbb C)$ subgroup. However,
 $\pi_\lambda$ is not isomorphic to $\pi_\mu^\vee$.

 In its failure to determine an irreducible representation of $\G$ (up to an outer automorphism),
the restriction problem to the principal $\SL_2(\C)$  contributes to a well-known question:

 \begin{enumerate}
 \item Are there irreducible representations of a simple group $\G$   (which are distinct even
   up to an outer automorphism) of the same dimension?

   \item Are there any constructions of irreducible representations of a simple group $\G$ of the same dimension?
   \end{enumerate}

 For the case of $\GL_n(\C)$ where irreducible representations are parameterized by $n$-tuple of (dominant) integers
 $\lambda = \{\lambda_1 \geq \lambda_2 \geq \cdots \lambda_n\}$, we take $\rho= \{n-1 > n-2 > \cdots > 1 > 0\}$,
 and $\lambda+\rho = \{l_1>l_2> \cdots > l_n\}$.

 In this case the question about which irreducible representations of $\SL_n(\C)$ have the
 same restriction to the principal $\SL_2(\C)$ becomes (by Theorem \ref{prod}):

 \vspace{5mm}
 
 {\bf Question:} Given an $n$-tuple of distinct integers  $X= \{l_1>l_2> \cdots > l_n\}$, if one knows
 the (multiset) set $X-X = \{ l_i-l_j| 1 \leq i,j \leq n\}$, does it determine $X$ up to translation:
 $X\rightarrow X+a$,
 and reflection $X\rightarrow b-X$ for integers $a,b$? Said differently, given subsets $X,Y$ of $\Z$ of cardinality
 $n$, such that
 \[ X-X = Y-Y,\]
 is $X= Y+b$, or $X=a-Y$ for some $a,b \in \Z$?

 \vspace{5mm}

 To make a different formulation of this question, associate a Laurent polynomial $f_X (z)$  in $z$
 to any set $X$ of integers as follows,
 \[f_X (z)= \sum  z^{l_i},\]
where $l_i$ are the integers in the set $X$.

As,
$f_{X-X}(z) = f_X(z) f_X(z^{-1}),$
to get the equality of  the difference sets $X-X = Y-Y$,
all we want is:
\[ f_{X-X}(z) = f_{Y-Y}(z) =  f_X(z) f_X(z^{-1}) =  f_Y(z) f_Y(z^{-1}).\]

Thus, given $X$, one needs to factorize the Laurent polynomial 
$f_X(z) f_X(z^{-1})$, which is essentially the same as to factories  $f_X(z)$ over integers.
In particular, if $f_X(z)$ is an irreducible Laurent polynomial, then
$X=Y$ up to the equivalence we have imposed on our sets (translation and reflection).

We illustrate this point of view with the example for $n =6$ given earlier.
In this case

\begin{eqnarray*}
  X & = &  \{11, 10, 9, 5, 3, 0\}, \\
Y & = & \{11,10, 5,4, 2,0\}, \\
f_X(z) & = & z^{11}+z^{10}+z^9+z^5+z^3+1= (z^2+1)(z^4+z^3+1)(z^5-z^2+1),\\
f_Y(z) & = & z^{11}+z^{10}+z^5+z^4+ z^2+ 1= (z^2+1)(z^4+z^3+1)(z^5-z^3+1).
\end{eqnarray*}

It is clear to see that:
\[f_X(z) f_X(z^{-1}) =  f_Y(z) f_Y(z^{-1}),\]
where the important thing to notice is that
$(z^5-z^2+1)$  and $(z^5-z^3+1)$ are related to each other by $g(z) \rightarrow z^5  g(z^{-1})$.

 Thus given a factorization of $f_X(z)$ as a product of irreducible Laurent polynomials in $z$ as $\prod f_i$,
we take $f_Y$ to be $ \prod g_i$,
where $f_i = g_i$ or 
$f_i = z^{\deg f_i} g_i(z^{-1})$, with the extra requirement that the polynomial $f_Y$
so obtained have  integral coefficients
which are either 0 or 1, a condition which is not easy to interpret,
will give rise to the desired equality $f_X(z) f_X(z^{-1}) =  f_Y(z) f_Y(z^{-1})$.

In any case, here is a construction of sets $X,Y$ with $X-X=Y-Y$, giving rise to representations of $\GL_n(\C)$
with isomorphic restrictions to $\SL_2(\C)$, in particular with the same dimensions.

Let $A,B$ be sets of integers such that $A+B$, and hence $A-B$, are sets of integers (instead of multisets),
\begin{eqnarray*}
X& = & A+B,\\
Y &=&  A-B,
\end{eqnarray*}
has the property:
\begin{eqnarray*}
X-X &=& A+B-A-B,\\
Y-Y &=&  A-B-A+B = X-X,
\end{eqnarray*}
giving a family of examples with $X-X=Y-Y$.

On the other hand, for $X=A+B$, $|X|=|A| |B|$, and hence the above construction does not yet
give the degree 6 example given above which had a factorization as a 
product 

\[f_X(z)= z^{11}+z^{10}+z^9+z^5+z^3+1= (z^2+1)(z^4+z^3+1)(z^5-z^2+1),\]
which cannot be of the form $f_B f_A$ because of the presence of negative coefficients in
$(z^5-z^2+1)$.

\newpage 

\part*{\centering
  {Part B: Centralizers of Finite order elements}}
\section{Review of Kac-coordinates} \label{Kac}
Let $\G$ be a simple and simply connected complex algebraic group.
Let $\B$ be a Borel subgroup of $\G$ and let $\T$ be a maximal torus of $B$.
Let $X^\ast(\T)$ be the character group of $\T$ and let $X_\ast(\T)$ be the
co-character group of $\T$. Let $\Phi$ be the set of roots of $\G$ with
respect to $\T$, and let $\{\alpha_i: 1\leq i\leq \ell\}$ be the set of simple roots
for the choice of $\B$.  We follow Bourbaki (see \cite[Planches I-VII, Chapter
4-6]{bourbaki_lie_4_6}) to index the set of simple roots.
Let $\alpha_0=-\sum_{i=1}^\ell a_i\alpha_i$ be the negative of the highest
root for the choice of $\T$ and $B$.  Let $a_0=1$ and note that
$a_0+a_1+\cdots+a_\ell$ is the Coxeter number of $\G$, denoted by $h$.

For a positive integer $m$, let $(s_0, s_1,s_2,\dots, s_\ell)$ be a tuple
of non-negative integers such that
        \begin{equation}\label{kac-co}
            \sum_{i=0}^\ell a_is_i=m\ 
        \text{and}\  (s_0, s_1,\dots, s_\ell)=1.
      \end{equation}
        For any such tuple $s=(s_0, s_1,\dots, s_\ell)$, there 
        exists a unique element $\tilde{x}(s)$ in $\T$ such that,
        for all $\chi\in X^\ast(\T)$ with $\chi=\sum_{i=1}^\ell c_i\alpha_i$ and $c_i\in \mathbb{Q}$, we have
        $$\chi(\tilde{x}(s))={\rm exp}\left(\dfrac{2\pi i\sum_{i=1}^\ell c_is_i}{m}\right).$$
Note that for any simple root $\alpha_k$, for $1\leq k\leq \ell$, we have
$$\alpha_k(\tilde{x}(s))={\rm exp}\left(\dfrac{2\pi i s_k}{m}\right),$$
and the above equality, for all simple roots $\alpha_k$, determines
$\tilde{x}(s)$ uniquely when $Z(\G)$ is trivial.

Kac showed that any finite order element in $\G$ is conjugate to a
unique element in $\T$ of the form $\tilde{x}(s)$, for some tuple
$(s_0, s_1,\dots, s_\ell)$ satisfying \eqref{kac-co} and for some $m$
(see \cite{serre_kac} and \cite{reeder_torsion} for an exposition and
further references).  The tuple $(s_0, s_1,\dots, s_\ell)$ is called
the Kac-coordinates of the element $\tilde{x}(s)$ as well as of the
corresponding elements in any quotient of $\G$ under a finite central
subgroup.  The centraliser of $\tilde{x}(s)$ is a connected reductive
subgroup of $\G$ containing $\T$ and a basis of its root system is
given by
$$\{\alpha_i:s_i= 0, 0\leq i\leq \ell\}.$$
In section \ref{exceptional_groups}, we
will use this to prove certain uniqueness result on $m$-torsion
elements having the minimal dimensional centraliser among all
$m$-torsion elements, when $m|h$.

Let $C$ be the set of all $(s_0, \dots, s_\ell)\in \mathbb{R}^{\ell+1}$
such that
$$s_i\geq 0,\ 0\leq i\leq \ell$$
and
$$a_0s_0+a_1s_1+a_2s_2+\cdots+a_\ell s_\ell=m.$$
Let $\Omega$ be the  group $X^\ast(T)/\mathbb{Z}\Phi$. There is a
natural action of $\Omega$ on the (extended) Dynkin
diagram on the vertices $\{\alpha_0, \alpha_1, \alpha_2,\dots, \alpha_\ell\}$.
Note that $C$ has a natural action of $\Omega$ by permuting the vertices of the
(extended) Dynkin diagram.

We denote by
$x(s)$ the image of $\tilde{x}(s)$ in $\G/Z(\G)$. For any $s$ satisfying the conditions
\eqref{kac-co} the order of $x(s)$
 in $\G/Z(\G)$
is $m$,
and any
element of order $m$ in $\G/Z(\G)$ is conjugate to $x(s)$, for some $s\in C$.
Two elements $x(s_1)$ and $x(s_2)$ are conjugate in $\G/Z(\G)$ if and only if
$s_1$ and $s_2$ belong to an orbit  $\Omega$ in $C$.

The following theorem of Kostant follows as a consequence of the above results. 
\begin{theorem}
          Let $\G$ be a complex simple algebraic group of adjoint type.
          Let $h$ be the Coxeter number of $\G$. There exists a unique 
          conjugacy class consisting of regular elements of order $h$. 
          \end{theorem}
          \begin{proof}
            Let $(s_0, s_1,\dots, s_\ell)$ be a tuple of Kac-coordinates
            of a regular element of order $h$ in $\T$. Note
            that $$\sum_{i=0}^\ell a_is_i=h=\sum_{i=0}^\ell a_i.$$ Since
            $x(s)$ is regular, $s_i\neq 0$, for all $1\leq i\leq
            l$. Thus, we get that $s_i=1$, for all $1\leq i\leq
            l$. This implies that any two regular elements of order
            $h$ are conjugate in $\G$.
            \end{proof}

            \section{Classical groups}
            Recall that we are using $h$ to denote  the Coxeter number
            of a reductive group $\G$, and $C_h$ the Coxeter
            conjugacy class in $\G$, defined as the
            unique conjugacy class in $[\G,\G]$ which is regular and
            of minimal order ($h$, or $2h$), with order $h$ in the
            adjoint group of $\G$.

In the following sections we prove the following theorem.

\begin{thm}
For any classical group
$\G=\GL_n(\C), \Sp_{2n}(\C)$, $\SO_{2n+1}(\C), \SO_{2n}(\C)$, and for any $d| h$, the Coxeter number of $\G$,
there is a unique conjugacy class containing an element $x_d$ of order $d$ in the adjoint group
of minimal dimensional centralizer $Z_\G(x_d)$ (among elements $x_d$ of order $d$ in the adjoint group),
which is represented by $C_h^{h/d}$, except in the following cases:

\begin{enumerate}

\item $\G=\SO_{2n+2}$,  $d| (h=2n)$, an odd integer $\geq 3$. In this case there is exactly one more
  conjugacy class besides the one represented by  $C_h^{2n/d}$.

  \item $\G=\SO_{2n+2}$,  $d| (h=2n)$, an even integer with $2n/d$ also even. In this case there is exactly one more
  conjugacy class besides the one represented by  $C_h^{2n/d}$.
\end{enumerate}
The precise conjugacy classes and their centralizers are explicitly given in Lemma \ref{odd}, and Lemma \ref{even}.

\end{thm}

In what follows, we will deal with each classical group separately. What is involved is not so much that we are
dealing with the centralizer of elements $x_d$ in $\G$ which are of order $d$ in the adjoint group but that we are dealing with
the centralizer of elements $x_d$ in $\G$ which have at most $d$ eigenvalues in the standard representation of the classical
group (although $x_d$ itself may have order $2d$, but having order $d$ in the adjoint group implies this),
and which if the classical group is not $\GL_n(\C)$, appear in pairs $\{\lambda, \lambda^{-1}\}$.  Thus all that we will use about
$x_d$ is that it has almost $d$ eigenvalues which can be any of the $d$-th, or $(2d)$-th but not $d$-th roots of unity, and 
we are free to use these eigenvalues. Further, we will use the fact that a conjugacy class (say of finite order) in a classical group $\G$ is
uniquely specified by its eigenvalues, except that for $\SO_{2n}(\C)$, one must use conjugacy  not by $\SO_{2n}(\C)$ but by
$\O_{2n}(\C)$. This issue makes no difference  for $\SO_{2n}(\C)$ when dealing with those elements one of whose eigenvalues is $\{1,-1\}$, but may affect otherwise. However, as all the elements we will eventually deal with, with minimal dimensional
centralizer, will have one of the eigenvalues $\{1,-1\}$, so  we will not concern ourselves with this distinction
between conjugacy classes in $\SO_{2n}(\C)$ and in
$\O_{2n}(\C)$. These considerations allow us, for example, to replace
a Levi subgroup
$\G(W) \times \GL(X)$ corresponding to the decomposition
$V = W+X+X^\vee$, which arises as part of the centralizer of an element $x_d$ with $X$ and $X^\vee$ isotropic, and perpendicular to
$W$, by any other decomposition $V = W'+X'+X'^\vee$, changing $X$ to $X'$ to decrease the dimension of the corresponding Levi
$\G(W') \times \GL(X')$ without changing the eigenvalues of $x_d$ involved.

\section{$\GL_n(\C)$}

Let $\omega_d = e^{2\pi i/d}$, and let $C_d$ be the principal conjugacy class in $\GL_n(\C)$
represented by:
\[ C_d = \begin{pmatrix}
  1 & & & & \\
  & \omega_d &  & &  \\
    &  & \ddots & &  \\
  & &  & \omega_d^{n-2} & \\
  & & & & \omega_d^{n-1}
  \end{pmatrix}. \]

In particular, $C_n$ is the Coxeter conjugacy class in $\GL_n(\C)$
represented by:
\[ C_n = \begin{pmatrix}
  1 & & & & \\
  & \omega_n &  & &  \\
    &  & \ddots & &  \\
  & &  & \omega_n^{n-2} & \\
  & & & & \omega_n^{n-1}
\end{pmatrix}. \]

For $d|n$, we have $C_d = (C_n)^{n/d}$, which is represented by:
\[  \begin{pmatrix}
  1 & & & & & & \\
    &   \ddots & & & & & \\
  &  & 1 & & & &\\
    &  & &  \ddots & & & \\
  & &  & & \omega_d^{d-1}&  & \\
& &  & &  & \ddots & \\
  & & & & & & \omega_d^{d-1}
  \end{pmatrix}, \]
where each eigenvalue $\omega_d^i$ appears with multiplicity $n/d$. Therefore,
the centralizer:
\[ Z_{\GL_n(\C)}(C_d) = \GL_{n/d}(\C)^d \subset \GL_n(\C).\]

On the other hand if $x_d$ is any element of order $d$ in $\GL_n(\C)$, thus with eigenvalues contained in $\{1,\omega_d, \omega_d^2,\cdots, \omega_d^{d-1}\}$,
the centralizer:
\[ Z_{\GL_n(\C)}(x_d) = \prod_{i=0}^{d-1} \GL_{n_i}(\C) \subset \GL_n(\C), {\rm ~~with~~}
\sum_{i = 0}^{d-1}n_i = n,\]
where $n_i$ (possibly zero) is the multiplicity of the eigenvalue $\omega_d^i$.

It follows that
\[ \dim Z_{\GL_n(\C)}(x_d) = \sum_{i=0}^{d-1} n_i^2  {\rm ~~with~~}
\sum_{i = 0}^{d-1}n_i = n.\]

Note  that if $(a+1) < b$, 
\[ a^2+b^2 > (a+1)^2 +(b-1)^2.\]
This means that for \[ \dim Z_{\GL_n(\C)}(x_d)= \sum_{i=0}^{d-1} n_i^2  {\rm ~~with~~}
\sum_{i = 0}^{d-1}n_i = n,\] to have the minimal possible dimension among the centralizer of elements with
at most $d$ eigenvalues,
\[|n_i-n_j| \leq 1 ~~~~~~~ \forall i,j.\]

Thus, we can assume that $\{n_i\} = \{a,\cdots, a, a+1,\cdots, a+1\}$,
with $a$ appearing $(d-b)$ times and $(a+1)$ appearing $b$ times. Therefore,
\[ \sum_{i = 0}^{d-1} n_i = n = ad + b.\]
When $d|n $, we get that $d|b$, giving us either $b=0$, or $b=d$, thus all
the $n_i$ are the same and equal to $n/d$, and $x_d$ is conjugate to $C_d$.

We summarize the conclusions drawn here in the following proposition (in which we allow $n, d$ to be arbitrary).

\begin{prop}
For 
$\G=\GL_n(\C)$, let  $d \geq 1$ be any integer $\leq n$. Then if $n \equiv 0,\pm 1 \bmod d$, there is a unique conjugacy class in the adjoint group
$\PGL_n(\C)$
containing an element $x_d$ of order $d$ in $\PGL_n(\C)$
of minimal dimensional centralizer $Z_\G(x_d)$ (among elements $x_d$ of order $d$ in $\PGL_n(\C)$),
which is represented by $C_d$. For an arbitrary value of $d$, writing $n=da+b$ with $0\leq b < d$,
although the conjugacy class of $x_d$ is not unique
$\PGL_n(\C)$, its centralizer is unique up to conjugacy, and is the Levi subgroup:
\[ Z_\G(x_d) = \GL_a(\C)^{d-b} \times  \GL_{a+1}(\C)^b.\]
  \end{prop}

\begin{proof}
  If $n \equiv b \bmod d$, then the elements of order $d$ with minimal dimensional centraliser
  have eigenvalues $\{\omega_d^i, i =0,\cdots, n-1\}$, of which $(d-b)$ are repeated  $a$ times, and
  $b$ of them are repeated   $a+1$ times. Thus when $b = 1$, or $(d-1)$, exactly one of the eigenvalues
  occurs with different multiplicity than the rest. As we are in $\PGL_n(\C)$, which of the  eigenvalues $\{\omega_d^i, i =0,\cdots n-1\}$ is singled out, does not matter.
  \end{proof}
\section{$\Sp_{2n}(\C)$}
Let $C_{2n}$ be the Coxeter conjugacy class in $\Sp_{2n}(\C)$
(corresponding to the symplectic form $X_1\wedge X_{2n} + X_2\wedge X_{2n-1}+\cdots$) 
of order $4n$ and of order $2n$ in the adjoint group, given by:

\[  C_{2n} = \begin{pmatrix}
  \omega_{4n}^{2n-1} & & && & && \\
    &   \ddots & & & & & &\\
  &  & \omega_{4n}^3 & & & & &\\
    &  & &  \omega_{4n} & & & &\\
  & &  & & \omega_{4n}^{-1} & & &\\
& &  & & & \omega_{4n}^{-3} & &\\
  & &  & & & & \ddots & \\
&  & & & & & & \omega_{4n}^{-2n+1}
\end{pmatrix}. \]

Let $C_{2n}^a=C_{2b}$ be an element of even order $2b$ in the adjoint group
for $2ab=2n$, thus $C_{2b}$ is given by the matrix:
\[ C_{2b}=  \begin{pmatrix}
  \omega_{4b}^{2n-1} & & && & && \\
    &   \ddots & & & & & &\\
  &  & \omega_{4b}^3 & & & & &\\
    &  & &  \omega_{4b} & & & &\\
  & &  & & \omega_{4b}^{-1} & & &\\
& &  & & & \omega_{4b}^{-3} & &\\
  & &  & & & & \ddots & \\
&  & & & & & & \omega_{4b}^{-2n+1}
\end{pmatrix}. \]

Then 
\[ Z_{\Sp_{2n}(\C)}(C_{2b}) = \GL_{a}(\C)^b \subset \Sp_{2n}(\C) = \Sp_{2ab}(\C).\]

If instead of $C_{2b}$, we take any element $x_{2b}$ of order $4b$
in $\Sp_{2n}(\C)$ which is of order $2b$ in the adjoint group, then $x_{2b}^{2b}=-1$, thus the eigenvalues of $x_{2b}$ must be $(4b)$-th root of unity, none of which are $(2b)$-th root of unity.
In particular $1,-1$ are not possible eigenvalues of $x_{2b}$,
and the number of eigenvalues of $x_{2b}$ are at most $2b$ many, occurring in pairs $\{\lambda, \lambda^{-1}\}$.
By replacing $\lambda$ by $\lambda^{-1}$, assume that $\lambda^b = i$, for pairs of eigenvalues  $\{\lambda, \lambda^{-1}\}$ of $x_{2b}$.

Clearly, the centralizer of $x_{2b}$ in $\Sp_{2n}(\C)$ is 
$\prod \GL(W_\lambda)$,
where $W_{\lambda}$ is the $\lambda$ eigenspace of $x_{2b}$, and the product over  $\lambda$ is taken over only one of $\lambda$ or $\lambda^{-1}$.

Thus, \[Z_{\Sp_{2n}(\C)} (x_{2b})  =
\prod \GL(W_\lambda), ~~{\rm with} ~~ \sum \dim(W_{\lambda}) = n.\]

Since we are in the case of $2ab=2n$, $b|n$ with $x_{2b}^b=i$, giving us an element of order $b$ in the adjoint group of $\GL_n(\C)$ with $b|n$, By the analysis done for $\GL_n(\C)$, for $Z_{\Sp_{2n}(\C)} (x_{2b})$ to have minimal possible dimension, $\dim(W_\lambda)$ is independent of $\lambda$, and $x_{2b}$ is a conjugate of $C_{2b}$, proving the assertion about $\Sp_{2n}(\C)$ for those elements of even order in the adjoint group.

Next we assume that $x$ has odd order $=(2b+1)$ in the adjoint group of $\Sp_{2n}(\C)$ with minimal possible dimension of its centralizer (among centralizer of elements of $\Sp_{2n}(\C)$ of order $=(2b+1)$ in the adjoint group). In this case,we can assume that  $x$ itself has odd order in $\Sp_{2n}(\C)$. Therefore, eigenvalue 1 is allowed, but not the eigenvalue $-1$ for $x$. In this case, we have the decomposition of the underlying symplectic space $W$ (of dimension $2n$) as:
\[ W = \sum_{ \lambda^2\not = 1}
(W_\lambda +W_{\lambda^{-1}}) \oplus W_1.\]
  The centralizer of $x$  in $\Sp(W)= \Sp_{2n}(\C)$ is:
  \[\prod \GL(W_\lambda)
  \times \Sp(W_1).\]

  We next need some lemmas. In the lemmas below, for a classical group $\G(W)$ defined using a quadratic or a symplectic space $W$, we will use  Levi subgroups of $\G(W)$ of the form 
$\G(W_1) \times \GL(X)$ preserving  the decomposition
  $W = X+W_1+ X^\vee$,
    where $X$ and $X^\vee$ are isotropic subspaces of $W$, and perpendicular to
    $W_1$, a non-degenerate subspace of $W$. In the case of the symplectic group $\Sp(W)$,
    these are particular cases of the 
subgroup $\prod \GL(W_\lambda)
\times \Sp(W_1)$ of $\Sp(W)$ just encountered for the  decomposition of $ W = \sum_{ \lambda^2\not = 1}
(W_\lambda +W_{\lambda^{-1}}) \oplus W_1.$

  \begin{lemma}\label{Sp}
    Among the maximal Levi subgroups of a symplectic group, the minimal dimension is achieved in exactly the following cases:

    \begin{enumerate}
    \item $\GL_{2n}(\C) \times \Sp_{2n}(\C) \subset \Sp_{6n}(\C)$.

    \item $\GL_{2n \pm 1}(\C) \times \Sp_{2n}(\C) \subset \Sp_{6n\pm 2}(\C)$.

    \item $\GL_{2n+2}(\C) \times \Sp_{2n}(\C) \subset \Sp_{6n+4}(\C)$.

      \end{enumerate}
In particular, there are two maximal Levi subgroups
      $\GL_{2n+2}(\C) \times \Sp_{2n}(\C)$ and $\GL_{2n+1}(\C) \times \Sp_{2n+2}(\C)$ inside
      $\Sp_{6n+4}(\C)$
      of (the same)
      minimal dimension.
  \end{lemma}

  For use in later sections, we record the following analogous lemma for odd
  orthogonal groups.

  \begin{lemma}\label{Bn}
    Among the maximal Levi subgroups of odd orthogonal groups, the minimal dimension is achieved in exactly the following cases:

    \begin{enumerate}
    \item $\GL_{2n}(\C) \times \SO_{2n+1}(\C) \subset \SO_{6n+1}(\C)$.

    \item $\GL_{2n \pm 1}(\C) \times \SO_{2n+1}(\C) \subset \SO_{6n+ 1  \pm 2}(\C)$.

    \item $\GL_{2n+2}(\C) \times \SO_{2n+1}(\C) \subset \SO_{6n+5}(\C)$.

      \end{enumerate}
In particular, there are two maximal Levi subgroups
      $\GL_{2n+2}(\C) \times \SO_{2n+1}(\C)$ and $\GL_{2n+1}(\C) \times \SO_{2n+3}(\C)$ inside
      $\SO_{6n+5}(\C)$
      of (the same)
      minimal dimension.
  \end{lemma}
  Here is the analogous lemma for even
  orthogonal groups.
  
   \begin{lemma} \label{Dn}
     Among the maximal Levi subgroups of an even orthogonal group,
     the minimal dimension is achieved in exactly the following cases:

    \begin{enumerate}
    \item $\GL_{2n}(\C) \times \SO_{2n}(\C) \subset \SO_{6n}(\C)$.

    \item $\GL_{2n \pm 1}(\C) \times \SO_{2n}(\C) \subset \SO_{6n\pm 2}(\C)$.

    \item $\GL_{2n-2}(\C) \times \SO_{2n}(\C) \subset \SO_{6n-4}(\C)$.
         \end{enumerate}
   In particular, there are two maximal Levi subgroups
      $\GL_{2n+1}(\C) \times \SO_{2n}(\C)$ and $\GL_{2n}(\C) \times \SO_{2n+2}(\C)$ inside
      $\SO_{6n+2}(\C)$
      of (the same)
      minimal dimension. 
   \end{lemma}

  Next observe that if 
\[Z(x) = \prod_{\lambda \sim \lambda^{-1}} \GL(W_\lambda)
  \times \Sp(W_1),\]
  has minimal possible dimension, then  $\prod \GL(W_\lambda)$ has minimal possible dimension among all Levi subgroups of $\GL(\sum W_\lambda)$ with as many factors, which implies by an earlier observation made during the analysis of $\GL_n(\C)$ that
  \[ |\dim W_\lambda - \dim W_\mu| \leq 1 {\rm ~~for~~ all~~} \lambda, \mu.\]
  Further, $\Sp(W_1) \times \GL(W_\lambda)$
  too must be of minimal dimension among maximal Levi subgroups of the symplectic group  $\Sp(W_1 + W_\lambda + W_{\lambda^{-1}} )$,
  for each $\lambda$. Therefore by  Lemma \ref{Sp},
 if $ \dim W_1 = 2m$, then  $ \dim W_\lambda = \dim W_{\lambda^{-1}} = 2m+a_\lambda$ with $a_\lambda = 0, \pm 1, 2$.

 Since $\sum (2m +a_i) = 2n$, it follows that,
 \[ 2m(2b+1) + \sum a_i = 2n.\]
 As the order of $x$ is $(2b+1)$ which divides the Coxeter number $2n$, we get that \[(2b+1)| \sum a_i.\]

Observe that the index $i$ of $a_i$ run over all the eigenvalues of $x$ which are $(2b+1)$ many.  As $a_i= 0, \pm 1, 2$, the options for $\sum a_i$ are  $\sum a_i= -(2b+1),0, (2b+1), 2(2b+1)$. Since $a_1=0$,  $\sum a_i=  -(2b+1), 2(2b+1)$ are not options,
 leaving  $\sum a_i= 0, (2b+1)$, as the only options. Given the further restriction that $|\dim W_\lambda - \dim W_\mu| \leq 1 {\rm ~~for~~ all~~} \lambda, \mu$, if some $a_i<0$, then no $a_i>0$, and $\sum a_i$ cannot be $(b+1)$. Assume next
 that all $a_i \geq 0$. If there is at least one $a_i=2$, then there are at least two corresponding to
 $\{\lambda, \lambda^{-1}\}$, and  none can be zero, so the only options are $a_i=1,2$ which together with   $\sum a_i= (2b+1),$ leads to a contradiction, leaving
 $\sum a_i = 0$ as the only option. In this case, if any $a_i<0$, then no $a_i>0$, and $\sum a_i =0$ is not an option, leaving us with the only option of $a_i=0$ for all $i$, i.e.,
 \[ \dim W_\lambda = \dim W_1,\]
 for all eigenvalues $\lambda$ of $x$ which means that the eigenvalues with multiplicity of $x$ is the same as that of $C_d$, and therefore they are conjugate
 in $\Sp_{2n}(\C)$.
 
We summarize the conclusions drawn above in the following proposition.

\begin{prop}
For 
$\G=\Sp_{2n}(\C)$, and for any $d|(2n)$, there is a unique conjugacy class in the adjoint group of
$\Sp_{2n}(\C)$
containing an element $x_d$ of order $d$ in the adjoint group
of minimal dimensional centralizer $Z_\G(x_d)$ (among elements $x_d$ of order $d$ in the adjoint group),
which is represented by $C_d$. The  centralizer $Z_\G(C_d)$ is always a Levi subgroup, and is

\begin{eqnarray*} Z_\G(x_d) & = & \GL_a(\C)^{d/2} \hspace{3cm} {\rm if}\hspace{1cm} 2n = da {\rm ~with~} d {\rm ~even~},\\
 Z_\G(x_d) &=& \GL_{2a}(\C)^{(d-1)/2} \times \Sp_{2a}(\C) \hspace{5mm} {\rm if}\hspace{1cm} 2n = 2da {\rm ~with~} d {\rm ~odd~}.\end{eqnarray*}
\end{prop}

 \section{$\SO_{2n+1}(\C)$}

 Let $C_{2n}$ be the Coxeter conjugacy class in $\SO_{2n+1}(\C)$ of order $2n$
 given by:

\[ C_{2n}= \begin{pmatrix}
  \omega_{2n}^{n} & & && & &&& \\
    &   \ddots & & & & & &&\\
  &  & \omega_{2n}^2 & & & && &\\
  &  & &  \omega_{2n} & & & & &\\
&  & &  & 1 & & & && \\
    & &  & & & \omega_{2n}^{-1}  && &\\
& &  & & & & \omega_{2n}^{-2} & &\\
  & &  & & & & & \ddots &\\
&  & & & & & & & \omega_{2n}^{-n}
\end{pmatrix}. \]

For $d|2n,$ let $d'=(2n)/d$, and let $C_d= C_{2n}^{2n/d}$ be an element of order $d$,
thus $C_{d}$ is the matrix:

\[ C_d=  \begin{pmatrix}
  \omega_{d}^{n} & & && & &&& \\
    &   \ddots & & & & & &&\\
  &  & \omega_{d}^2 & & & && &\\
  &  & &  \omega_{d} & & & & &\\
&  & &  & 1 & & & && \\
    & &  & & & \omega_{d}^{-1}  && &\\
& &  & & & & \omega_{d}^{-2} & &\\
  & &  & & & & & \ddots &\\
&  & & & & & & & \omega_{d}^{-n}
\end{pmatrix}. \]

Then for $d$ odd,
\[ Z_{\SO_{2n+1}(\C)}(C_{d}) = \SO_{d'+1}(\C) \times \GL_{d'}(\C)^{(d-1)/2} \subset \SO_{2n+1}(\C) = \SO_{dd'+1}(\C).\]

Note that for a general element $x_d$ of odd order $d$,
\[ Z_{\SO_{2n+1}(\C)}(x_{d}) = \SO_{2a_1+1}(\C) \times \prod \GL_{a_i}(\C) \subset \SO_{2n+1}(\C),\] in which there is a single $\GL_{a_i}(\C)$
corresponding to the eigenspaces $W_\lambda$ and $W_{\lambda^{-1}}$. Further, we have
\[ 2a_1+1 + 2 \sum_{i=2}^{d}  a_i = (2n+1),\]
this time there are separate terms $a_i=a_j$ corresponding to the eigenspaces $W_\lambda$ and $W_{\lambda^{-1}}$.

We need to prove that  $\dim Z_{\SO_{2n+1}(\C)}(x_{d}) = a_1(2a_1+1) +  \sum  a_i^2$ is minimal if and only if $x_d$ is a conjugate of $C_d$, thus if and only if $a_i = 2a_1$ for all $i \geq 2$.

Observe now that there is a bijection between those semisimple conjugacy classes in $\SO_{2n+1}(\C)$ and $\Sp_{2n}(\C)$ which do not have $-1$ as an eigenvalue by matching
dimensions of all eigenspaces not equal to 1, and for the eigenspace of dimension $(2d+1)$ of eigenvalue 1
for $\SO_{2n+1}(\C)$ matching with the eigenspace of dimension $(2d)$ for $\Sp_{2n}(\C)$ of eigenvalue 1.

Under this correspondence of semisimple conjugacy classes, centralizers match after changing $\SO_{2d+1}(\C)$ to $\Sp_{2d}(\C)$, thus they have the same dimension. Hence our theorem for $\Sp_{2n}(\C)$ for conjugacy classes of odd order diving $2n$ proves a similar theorem for $\SO_{2n+1}(\C)$ for conjugacy classes of odd order dividing
$2n$.

Next assume that $d$ is an even integer dividing $2n$, and $x_d$ an element
of $\SO_{2d+1}(\C)$ of order $d$. Write the centralizer in the form:
\[ Z_{\SO_{2n+1}(\C)}(x_{d}) = \SO_{d_0'}(\C)
\times \SO_{d_1'}(\C) \times  \prod \GL_{d_i}(\C) \subset \SO_{2n+1}(\C),\]
with $d_0'$ odd and $d_1'$ even and with
\[ d_0'+d_1' + \sum_{i=2}^{d}  d_i = (2n+1).\]

We need to prove that  $\dim Z_{\SO_{2n+1}(\C)}(x_{d})$  is minimal
if and only if $x_d$ is a conjugate of $C_d$, thus if and only if $a_i = 2a_1$ for all $i \geq 2$.

   Note also the following lemma.

   \begin{lemma} \label{son}
     Among the subgroups
     $ \SO_a(\C) \times \SO_{n-a}(\C)$ of $\SO_n(\C)$,
     the minimal dimension is achieved in exactly the following cases:
     \[ | a-(n-a)| \leq 1.\]
     Further, among the centralizers of elements of order 2 in
     $\SO_{2n+1}(\C)$,
     the minimal dimension is achieved by  $ \SO_n(\C) \times \SO_{n+1}(\C)$,
     whereas for $\SO_{2n}(\C)$,
     the minimal dimension is achieved either by  $ \SO_n(\C) \times \SO_{n}(\C)$,
 or     $ \SO_{n-1}(\C) \times \SO_{n+1}(\C)$, depending on whether $n$ is even or odd.
   \end{lemma}

Using Lemma \ref{son} and Lemma \ref{Bn},  
we find   that  if $\dim Z_{\SO_{2n+1}(\C)}(x_{d})$  is minimal
among all such subgroups of $\SO_{2n+1}(\C)$ with at most $d$ factors, then
for
\[ Z_{\SO_{2n+1}(\C)}(x_{d}) = \SO_{d_0'}(\C)
\times \SO_{d_1'}(\C) \times  \prod \GL_{d_i}(\C) \subset \SO_{2n+1}(\C),\]
we must be in one of the two cases:
\begin{enumerate}
\item $\SO_m(\C) \times \SO_{m+1}(\C) \times \GL_m(\C)^a \times \GL_{m+1}(\C)^b,$ or
  \item $\SO_m(\C) \times \SO_{m+1}(\C) \times \GL_{m-1}(\C)^a \times \GL_{m}(\C)^b.$ 
\end{enumerate}

In case (1),
adding up the dimensions of the various eigenspaces of $x_d$, we have,
\[ m + (m+1) + 2am + 2b(m+1) = 2n+1,\]
therefore, $2m(1+a+b) +(1+2b) =2n+1$.

On the other hand, $d$ is equal to the number of eigenvalues of $x_d$, and therefore, $d=2+2a+2b$, and the previous equation becomes $md+2b = 2n$.
As we are given that $d|2n$, we must have $d=2(1+a+b)| 2b$, which means $b=0$.

In case (2), adding up the dimensions of the various eigenspaces of $x_d$, we have,
\[ m + (m+1) + 2a(m-1) + 2bm = 2n+1,\]
therefore, $2m(1+a+b) +(1-2a) =2n+1$.

As before, $d$ is equal to the number of eigenvalues of $x_d$, and therefore, $d=2+2a+2b$, and the previous equation becomes $md-2a = 2n$.
As we are given that $d|2n$, we must have $d=2(1+a+b)| 2a$, which implies that $a=0$.

Thus the only option for the minimal dimensional centralizer is
\[ \SO_m(\C) \times \SO_{m+1}(\C) \times \GL_{m}(\C)^a ,\]
which completes the proof of the result for $\SO_{2n+1}(\C)$.

We summarize the conclusions drawn in this section.

\begin{prop}
For 
$\G=\SO_{2n+1}(\C)$, and for any $d|(2n)$, there is a unique conjugacy class in 
$\SO_{2n+1}(\C)$
containing an element $x_d$ of order $d$
of minimal dimensional centralizer $Z_\G(x_d)$,
which is represented by $C_d$. The  centralizer $Z_\G(C_d)$ is a Levi subgroup
if $d$ is odd, and is

\begin{eqnarray*} Z_\G(x_d) & = & \GL_a(\C)^{(d-2)/2}
  \times \SO_a(\C)\times \SO_{a+1}(\C) \hspace{5mm} {\rm if}\hspace{1cm} 2n = da {\rm ~with~} d {\rm ~even},\\
 Z_\G(x_d) &=& \GL_{2a}(\C)^{(d-1)/2} \times \SO_{2a+1}(\C) \hspace{2cm} {\rm if}\hspace{1cm} 2n = 2da {\rm ~with~} d {\rm ~odd}.\end{eqnarray*}
\end{prop}

 \section{$\SO_{2n+2}(\C)$}

 Let $C_{2n}$ be the Coxeter conjugacy class in $\SO_{2n+2}(\C)$ of order $2n$
 given by:

\[ C_{2n} =  \begin{pmatrix}
  \omega_{2n}^{n} & & && & &&&& \\
    &   \ddots & & & & & &&& \\
  &  & \omega_{2n}^2 & & & &&& & \\
  &  & &  \omega_{2n} & & & && & \\
  &  & &  & 1 &  & &&& \\
  &  & &  &  & 1  & & && \\
    & &  & & & & \omega_{2n}^{-1}  && & \\
& &  & & & & &\omega_{2n}^{-2} & &  \\
  & &  & & & & & & \ddots &  \\
&  & & & & & & & & \omega_{2n}^{-n}
\end{pmatrix}. \]

For $d|2n,$ let $C_d= C_{2n}^{2n/d}$ be an element of order $d$,
thus $C_{d}$ is the matrix:

\[  C_d = \begin{pmatrix}
  \omega_{d}^{n} & & && & &&&& \\
    &   \ddots & & & & & &&& \\
  &  & \omega_{d}^2 & & & &&& & \\
  &  & &  \omega_{d} & & & && & \\
  &  & &  & 1 &  & &&& \\
  &  & &  &  & 1  & & && \\
    & &  & & & & \omega_{d}^{-1}  && & \\
& &  & & & & &\omega_{d}^{-2} & &  \\
  & &  & & & & & & \ddots &  \\
&  & & & & & & & & \omega_{d}^{-n}
\end{pmatrix}. \]

The standard representation of  $\SO_{2n+2}(\C)$ treated as a representation of the cyclic group $\Z/(2n)$
generated by $C_{2n}$ is the sum of the regular representation of $\Z/(2n)$ with the regular representation of $\Z/2 = \Z/(2n)/2\Z/(2n)$.
Therefore, for $d|(2n)$, standard representation of  $\SO_{2n+2}(\C)$ treated as a representation of the cyclic group $\Z/d$
generated by $C_d$ is the sum of $d'=(2n)/d$ many copies of the regular representation of $\Z/d$ plus either the two dimensional trivial representation of $\Z/d$ or the regular representation of $\Z/2$ treated as a representation of $\Z/d$ if $d$ is even and $d'$ is odd.

Then for $d$ odd, $d'=(2n)/d$, we have:
\[ Z_{\SO_{2n+2}(\C)}(C_{d}) = \SO_{d'+2}(\C) \times \GL_{d'}(\C)^{(d-1)/2} \subset \SO_{2n+2}(\C) = \SO_{dd'+2}(\C).\]

On the other hand, for $d$ even, there are two cases depending on the parity of $d'=(2n)/d$:

\begin{enumerate}
  \item $d'$ odd:
\[ Z_{\SO_{2n+2}(\C)}(C_{d}) = \SO_{d'+1}(\C) \times \SO_{d'+1}(\C) \times \GL_{d'}(\C)^{(d-2)/2} \subset \SO_{2n+2}(\C) = \SO_{dd'+2}(\C).\]
\item $d'$ even:
\[ Z_{\SO_{2n+2}(\C)}(C_{d}) = \SO_{d'}(\C) \times \SO_{d'+2}(\C) \times \GL_{d'}(\C)^{(d-2)/2} \subset \SO_{2n+2}(\C) = \SO_{dd'+2}(\C).\]
  \end{enumerate}
Note that for a general element $x_d$ of odd order $d$,
\[ Z_{\SO_{2n+2}(\C)}(x_{d})
= \SO_{2a_1}(\C) \times \prod \GL_{a_i}(\C) \subset \SO_{2n+2}(\C),\]
with \[ 2a_1 + \sum_{i=2}^{d}  a_i = (2n+2).\]

Assume that $x_d$ is of odd order which is $d$.
By Lemma \ref{Dn}, if $\dim Z_{\SO_{2n+2}(\C)}(x_{d})$ is to assume the minimal dimension
among centralizer of elements of order $d$, then
we must be in one of the three cases:

\begin{enumerate}
\item $  \hspace{ 1cm} Z_{\SO_{2n+2}(\C)}(x_{d}) = \SO_{2m}(\C) \times \GL_{2m}(\C)^{a} \times \GL_{2m+1}(\C)^{b},$

\item $  \hspace{ 1cm} Z_{\SO_{2n+2}(\C)}(x_{d}) = \SO_{2m}(\C) \times \GL_{2m}(\C)^{a} \times \GL_{2m-1}(\C)^{b},$

  \item $  \hspace{ 1cm} Z_{\SO_{2n+2}(\C)}(x_{d}) = \SO_{2m}(\C) \times \GL_{2m-2}(\C)^{a} \times \GL_{2m-1}(\C)^{b}.$
  
  \end{enumerate}

We analyze these possibilities one by one, making essential use of the hypothesis that $d|2n$, the Coxeter number
of $\SO_{2n+2}(\C)$.

In case (1), we have,
\[ 2m+ 2a(2m) + 2b(2m+1) = 2n+2, i.e., 2m(1+2a+2b) + 2b = 2n+2.\]
As $d=(1+2a+2b) | 2n$, $d| (2b-2)$, so $b$ must be
1.

In case (2), we have,
\[ 2m+ 2a(2m) + 2b(2m-1) = 2n+2, i.e., 2m(1+2a+2b) - 2b = 2n+2.\]
As $d=(1+2a+2b) | 2n$, $d| (2b+2)$,
so $b = -1$, which is not an option, so case (2) does not happen.

In case (3), we have,
\[ 2m+ 2a(2m-2) + 2b(2m-1) = 2n+2, i.e., 2m(1+2a+2b) -4a- 2b = 2n+2.\]
As $d=(1+2a+2b) | 2n$, $d| (4a+2b+2)$,
but as $d$ is assumed to be odd, $d| (2a+b+1)$ with $d=1+2a+2b$, which means $b=0$ is the only option.

Thus, among elements $x_d$ of odd order $d$, the possibilities for 
$Z_{\SO_{2n+2}(\C)}(x_{d})$, assuming it to have the minimal dimension, is:

\begin{enumerate}
\item $  \hspace{ 1cm} Z_{\SO_{2n+2}(\C)}(x_{d}) = \SO_{2m}(\C) \times \GL_{2m}(\C)^{a} \times \GL_{2m+1}(\C),$

  \item $  \hspace{ 1cm} Z_{\SO_{2n+2}(\C)}(x_{d}) = \SO_{2m}(\C) \times \GL_{2m-2}(\C)^{a}.$
  
  \end{enumerate}

Case (2) has the same centralizer $Z_{\SO_{2n+2}(\C)}(x_{d})$  as for $C_d$, and it is easy to see that in this case $x_d$ and $C_d$ are conjugate. On the other hand, Case (1) also gives rise to centralizers of minimal dimension, which we make precise in the following lemma.

\begin{lemma} \label{odd}
  For $d\geq 3$, an odd integer, dividing $2n$, let $m = n/(2d)$. Then there is an element $x_d$ of order $d$ whose centralizer is
  \[  Z_{\SO_{2n+2}(\C)}(x_{d}) = \SO_{2m}(\C) \times \GL_{2m}(\C)^{\frac{d-3}{2}} \times \GL_{2m+1}(\C),\]
  whereas the centralizer of $C_d$ is:
  \[ Z_{\SO_{2n+2}(\C)}(C_{d}) = \SO_{2m+2}(\C) \times \GL_{2m}(\C)^{\frac{d-1}{2} }.\]
  The dimensions of the centralizers of $x_d$ and $C_d$ are the same, but they are not conjugate.
  \end{lemma}

Assume next that $x_d$ is of even order which is $d$.
By Lemma \ref{Dn}, we must be in one of the four cases:

\begin{enumerate}
\item $ \hspace{ 1cm} Z_{\SO_{2n+2}(\C)}(x_{d}) = \SO_{m}(\C) \times \SO_{m}(\C) \times
  \GL_{m}(\C)^{a} \times \GL_{m+1}(\C)^{b},$

\item $ \hspace{ 1cm} Z_{\SO_{2n+2}(\C)}(x_{d}) = \SO_{m}(\C) \times \SO_{m}(\C) \times \GL_{m}(\C)^{a} \times \GL_{m-1}(\C)^{b},$

\item $  \hspace{ 1cm}Z_{\SO_{2n+2}(\C)}(x_{d}) = \SO_{m}(\C) \times \SO_{m}(\C) \times \GL_{m-2}(\C)^{a} \times \GL_{m-1}(\C)^{b}.$
  \item $  \hspace{ 1cm} _{\SO_{2n+2}(\C)}(x_{d}) = \SO_{m+2}(\C) \times \SO_{m}(\C) \times \GL_{m+1}(\C)^{a} \times \GL_{m}(\C)^{b}.$
  
  \end{enumerate}

We analyze these possibilities one by one.

In case (1), we have,
\[ 2m+ 2a(m) + 2b(m+1) = 2n+2, i.e., 2m(1+a+b) + 2b = 2n+2.\]
As $d=(2+2a+2b) | 2n$, $d| (2b-2)$, so $b$ must be
1.

In case (2), we have,
\[ 2m+ 2a(m) + 2b(m-1) = 2n+2, i.e., 2m(1+a+b) - 2b = 2n+2.\]
As $d=(2+2a+2b) | 2n$, $d = (2+2a+2b)| (2b+2)$,
so $a = 0$ is the only option.

In case (3), we have,
\[ 2m+ 2a(m-2) + 2b(m-1) = 2n+2, i.e., 2m(1+a+b) -4a- 2b = 2n+2.\]
As $d=(2+2a+2b) | 2n$, $d = (2+2a+2b)| (4a+2b+2)$,
which means $a=0$ is the only option, but case (2) and case (3) for $a=0$ are the same, so we need to consider only
one of them.

 In case (4), we have,
\[ 2m+2+ 2a(m+1) + 2bm = 2n+2, i.e., 2m(1+a+b) +2a+2 = 2n+2.\]
As $d=(2+2a+2b) | 2n$, $d = (2+2a+2b)| (2a)$,
 which means $a=0$ is the only option.

Thus, among elements $x_d$ of even order $d$, the possibilities for 
$Z_{\SO_{2n+2}(\C)}(x_{d})$,  when  of minimal dimension, are:

\begin{enumerate}
\item   $\hspace{ 1cm} Z_{\SO_{2n+2}(\C)}(x_{d}) = \SO_{m}(\C) \times \SO_{m}(\C) \times \GL_{m}(\C)^{a} \times \GL_{m+1}(\C),$

\item   $\hspace{ 1cm} Z_{\SO_{2n+2}(\C)}(x_{d}) = \SO_{m}(\C) \times \SO_{m}(\C) \times \GL_{m-1}(\C)^{b}, $
  
  \item  $ \hspace{ 1cm} Z_{\SO_{2n+2}(\C)}(x_{d}) =  \SO_{m+2}(\C) \times \SO_{m}(\C) \times \GL_{m}(\C)^{b}. $
    
  \end{enumerate}

Now, we express the above centralizer using $d|(2n)$, and $d'=(2n)/d$, and we find the 
possibilities for 
$Z_{\SO_{2n+2}(\C)}(x_{d})$,  when  of minimal dimension, are:

\begin{enumerate}
\item   $\hspace{ 1cm} Z_{\SO_{2n+2}(\C)}(x_{d}) = \SO_{d'}(\C) \times \SO_{d'}(\C) \times \GL_{d'}(\C)^{\frac{d}{2}-2} \times \GL_{d'+1}(\C),$

\item   $\hspace{ 1cm} Z_{\SO_{2n+2}(\C)}(x_{d}) = \SO_{d'+1}(\C) \times \SO_{d'+1}(\C) \times \GL_{d'}(\C)^{\frac{d}{2}-1}, $
  
  \item  $ \hspace{ 1cm} Z_{\SO_{2n+2}(\C)}(x_{d}) =  \SO_{d'+2}(\C) \times \SO_{d'}(\C) \times \GL_{d'}(\C)^{\frac{d}{2}-1}. $
    
  \end{enumerate}

It can be seen that the dimension of $Z_{\SO_{2n+2}(\C)}(x_{d})$ in cases (1) and (3) are the same, and that the case (2) has dimension
which is one less. However, we remind ourselves that if $(d'+1)$ is odd, then case (2) does not arise as the centralizer of an
element in $\SO_{2n+2}(\C)$. Therefore if $d'$ is even, then there are two options for 
$Z_{\SO_{2n+2}(\C)}(x_{d})$,  with the same minimal dimension, corresponding to cases (1) and (3), whereas if $d'$ is odd, there is
the only  option for 
$Z_{\SO_{2n+2}(\C)}(x_{d})$ for it to have the minimal dimension, corresponds to case (2).
We summarize the above discussion  in the following lemma.

\begin{lemma} \label{even}
  For $d$, an even integer, dividing $2n$, suppose $d' = 2n/d$ is also even.
  Then there is an element $x_d$ of order $d$ whose centralizer is
  \[  Z_{\SO_{2n+2}(\C)}(x_{d}) = \SO_{d'}(\C) \times \SO_{d'}(\C) \times \GL_{d'}(\C)^{\frac{d}{2}-2} \times \GL_{d'+1}(\C),\]
    whereas the centralizer of $C_d$ is:
    \[ Z_{\SO_{2n+2}(\C)}(C_{d}) = \SO_{d'}(\C) \times \SO_{d'+2}(\C)
      \times \GL_{d'}(\C)^{\frac{d}{2}-1}. \]
  
    The dimensions of the centralizers of $x_d$ and $C_d$ are the
    same, but they are not conjugate.

  If $d'=2n/d$ is odd, then $C_d$ is the only conjugacy class in
  $\SO_{2n+2}(\C)$ containing an element of order $d$ with minimal
  dimensional centralizer.
  \end{lemma}

  Next, let us assume that $x_d$ is of order $d$, an even integer, in
  the adjoint group of $\SO_{2n+2}(\C)$, but of order $2d$ in
  $\SO_{2n+2}(\C)$, thus $x_d^d=-1 \in \SO_{2n+2}(\C)$.  Of course,
  the analysis above continues to hold good except that we need to pay
  attention to the fact that although $x_d$ has order $2d$, we are
  comparing the dimension of its centralizer with all elements of
  $\SO_{2n+2}(\C)$ whose order in the adjoint group is $d$.

In this case, none of the eigenvalues of $x_d$ can be 1,$-1$, more
precisely, the eigenvalues of $x_d$ are $(2d)$-th roots of unity but
none are $d$-th roots of unity, which are $d$-many $(2d)$-th roots of
1, and under the equivalence $\lambda \sim \lambda^{-1}$, there are
$d/2$ many equivalence classes.  The centralizer of $x_d$ in
$\SO_{2n+2}(\C)$ is $\prod \GL(W_\lambda)$ where $W_\lambda$ are the
eigenspaces of $x_d$ where only one of $\lambda,\lambda^{-1}$ is taken
from all the eigenvalues of $x_d$, which gives at most $d/2$ many
$\lambda$'s in the centralizer which is $\prod \GL(W_\lambda)$.  We
have $\sum \dim (W_\lambda) = (n+1)$ with $d|2n$, or
$(d/2)|n$. Therefore the minimal dimension of $\prod \GL(W_\lambda)$
is achieved when all but one of $W_\lambda$ have the same dimension,
and the remaining $W_\lambda$ has one extra dimension.

By the description of the centralizer of $C_d$, to prove that the
dimension of the centralizer of $C_d$ is smaller than the dimension of
the centralizer of $x_d$, an element of order $(2d)$ in
$\SO_{2n+2}(\C)$ but of order $d$ in the adjoint group, we are reduced
to checking that (use the computation on $Z_{\SO_{2n+2}(\C)}(x_{d})$
mentioned earlier which had two cases depending on the parity of
$d'=(2n)/d$):
\[ \dim \GL_{2a}(\C) > \max \dim \{ \SO_a(\C)  \times \SO_a(\C),
\SO_{a+1}(\C) \times \SO_{a-1}(\C),\]
whose pleasant task we leave to the reader.
\section{Exceptional groups}\label{exceptional_groups}
            In this section, when $m|h$, we determine 
            the conjugacy classes of maximal dimension 
            among the conjugacy classes consisting of order $m$
            elements. 
        For the sake of brevity, the type 
    of centralizer of an element in $G$ will be the type of the Lie algebra of the
    derived group of the centralizer. 
\begin{theorem}
      Let $G$ be an adjoint simple group of exceptional type. 
      Let $h$ be the Coxeter number of $G$. Let $m|h$ and let 
      $x\in G$ be an element of order $m$. Assume that 
      $m\neq 4$ when $G$ is of type ${\rm E}_6$
      and $m\neq 9$ when $G$ is of type ${\rm E_7}$. Then there is a unique 
      conjugacy class of order $m$ elements which 
      has the largest dimension among the conjugacy classes of 
      order $m$ elements. Moreover, we have
      \begin{enumerate}
      \item When $G$ is of type ${\rm E}_6$ there are  three conjugacy
        classes, consisting of elements of order $4$,
        which attain the maximal
        dimension among conjugacy classes of order $4$ elements.
        One of these classes contains $C_h^3$ 
        whose centralizer is of type $2A_2+A_1$. 
        The other two conjugacy
        class contains elements with centralizer of type $A_3+A_1$. 
     \item When $G$ is of type ${\rm E}_7$ there are two conjugacy
        classes, consisting of elements of order $9$,
        which attain the maximal
        dimension among conjugacy classes of order $9$ elements. One
        of these classes contains $C_h^2$ whose centralizer is of type $4A_1$. 
        The other conjugacy
        class contains elements with centralizer of type $A_1+A_2$. 
        \end{enumerate}
      \end{theorem}
      \begin{proof}
      The proof is case by case analysis for each exceptional group. 
      We introduce
            some convenient notation.
        \begin{enumerate}
            \item Let $X(G, m)$ be the
        set of tuples $(s_0, s_1,\dots, s_l)$ which satisfy 
        \eqref{kac-co}.
        \item Let $d(s)$
        be the number of 
        roots in the root system of the centralizer of $x(s)$
        with respect to the chosen maximal torus $T$ containing $x(s)$.
        \item Let $d(G, m)$ be the integer $\min\{d(t), t\in X(G, m)\}$. 
        \end{enumerate}
        
          In the following table, for each $G$, and for $2\leq m<h$,
    we list elements $s$ in $X(G, m)$ such that $d(s)=d(G, m)$.
    For $s\in X(G, m)$ we denote by $s_i$ the $i$-th 
    coordinate in the tuple $s$.
  \begin{center}\label{table_cen}
      \begin{tabular}{| m{1cm} | m{2cm} | m{4cm} | m{2cm} |}
      \hline
      Type & Order ($m$)& Kac coordinates&Centralizer type\\
      \hline
      ${\rm G_2}$& $2$ & $(0, 1,0)$ &\ $2A_1$ \\
      \hline
      ${\rm G_2}$&$3$ &(1,0,1) &$A_1$  \\
      \hline
      ${\rm F_4}$&$2$ &(0,1,0,0,0) &$A_1+C_3$  \\
      \hline
      ${\rm F_4}$&$3$ &(0, 0, 1,0,0)&$2A_2$  \\
      \hline
      ${\rm F_4}$&$4$ &(1,0,1,0,0) & $A_1+A_2$ \\
      \hline
      ${\rm F_4}$&$6$ &$(1,0,1,0,1)$ &$2A_1$  \\
      \hline
      ${\rm E_6}$&$2$ &$(0,0,1,0,0,0,0)$ & $A_1+A_5$ \\
      \hline
      ${\rm E_6}$&3 &$(0,0,0,0,1,0,0)$ &$3A_2$  \\
      \hline
      ${\rm E_6}$&4 &$(0,1,0,0,1,0,0)$ 
       &  $2A_2+A_1$ \\
       ${\rm E_6}$&4 &(1,1,0,0,0,1,0) 
       & 
         $A_3+A_1$ \\
         ${\rm E_6}$&4 &(1,0,0,1,0,0,1) 
       & 
      $A_3+A_1$ \\
      \hline
      ${\rm E_6}$&6 &$(1,1,0,0,1,0,1)$ &$3A_1$  \\
      \hline
      ${\rm E_7}$&2 &$(0, 0,1,0,0,0,0,0)$ &$A_7$  \\
      \hline
      ${\rm E_7}$&$3$ &(0,0,0,1,0,0,0,0) &$A_2+A_5$\\
      \hline
      ${\rm E_7}$&6 &(1,0,0,0,1,0,0,1) & $2A_2+A_1$ \\
      \hline
      ${\rm E_7}$&9 &$(1,0,1,1,0,0,1,1)$ & $A_1+A_2$  \\
      ${\rm E_7}$&9&$(0, 1,0,0,1,0,1,1)$&$4A_1$\\
      \hline
      ${\rm E_8}$&2 &$(0,1,0,0,0,0,0,0,0)$ &$D_8$  \\
      \hline
      ${\rm E_8}$&3 &$(0,0,1,0,0,0,0,0,0)$ &$A_8$ \\
      \hline
      ${\rm E_8}$&5 &$(0,0,0,0,0,1,0,0,0)$ &$2A_4$  \\
      \hline
      ${\rm E_8}$&6 &$(1,0,0,0,0,1,0,0,0)$ &$A_4+A_3$  \\
      \hline
      ${\rm E_8}$&10 &$(1,0,0,0,1,0,0,1,0)$ &$2A_2+2A_1$  \\
      \hline
      ${\rm E_8}$&15 &$(1,1,0,0,1,0,1,0,1)$ &$4A_1$  \\
      \hline
    \end{tabular}
  \end{center}
  The case where $m=2$ and $G$ is of any type, we directly appeal to 
  Helgason \cite[Chapter X, no. 6, Table II, III]{Hel78}. 
  \subsection{Type ${\rm G}_2$}
  Assume that $G$ is of the type ${\rm G}_2$. Consider the case $m=3$. In this 
  case the equation \eqref{kac-co} becomes:
  $$s_0+3s_1+2s_2=3,\ (s_0, s_1, s_2)=1.$$
  The set $X(G_2, 3)$ contains two elements $(1, 0,1)$ and
  $(0,1,0)$. The corresponding centralizers are of type $A_1$ and
  $A_2$. Thus, there is a unique conjugacy class of order $3$ elements
  such that the dimension of the centralizer is $d(G, 3)$.
  \subsection{Type ${\rm F}_4$}
  Consider the case where $G$ is of type ${\rm F}_4$. In this case 
  \eqref{kac-co} is given by 
  $$s_0+2s_1+3s_2+4s_3+2s_4=m,\ (s_0, s_1, s_2, s_3, s_4)=1,$$
  where $m\in \{3,4,6\}$. 
  \subsubsection{}
  Let $s\in X(F_4, 3)$. We first note that $s_3=0$. If $s_2\neq 0$,
  then $s_2=1$ and $s_0=s_1=s_4=0$. So we have
  $(s_0, s_1, s_2, s_3, s_4)=(0,0,1,0,0)$. When $s_2=0$, then $s$ is
  either $(1, 1,0,0,0)$ or $(1,0,0,0,1)$. Thus $X(F_4, 3)$ contains
  three elements namely: $(0,0,1,0,0)$, $(1, 1,0,0,0)$ and
  $(1,0,0,0,1)$. The centralizers of the corresponding order $3$
  elements are of the type $2A_2$, $C_3$ and $B_3$ respectively.
  Thus, there exists a unique element $s\in X(F_4, 3)$ such that
  $d(s)=d(F_4, 3)$.
  
 \subsubsection{}
 Consider the case $m=4$ and in this case if $s_3\neq 0$, then
 $s_0=s_1=s_2=s_4=0$. This corresponds to the case $(0, 0, 0, 1,
 0)$. When $s_3=0$ and $s_2\neq 0$, we have $s_0=1, s_1=s_4=0$; this
 corresponds to $(1,0,1,0,0)$. Finally in the case where $s_3=s_2=0$,
 the possibilities are $(0, 1,0,0,1)$, $(2, 1, 0,0,0)$ and
 $(2, 0, 0,0, 1)$. Thus, $X(F_4, 4)$ contains: $(0, 0, 0, 1, 0)$,
 $(1,0,1,0,0)$, $(0, 1,0,0,1)$, $(2, 1, 0,0,0)$ and $(2, 0, 0,0,
 1)$. Thus the element $t=(1,0,1,0,0)$ is the unique element in
 $X(F_4, 4)$ such that $d(t)=d(F_4, 4)$

  \subsubsection{}
  Assume that $m=6$. In this case $t=(1,0,1,0,1)$ belongs to
  $X(F_4, 6)$, and centralizer of $x(t)$ is of type $2A_1$.  Let
  $s\in X(F_4,6)$ be an element such that $d(s)=d(F_4, 6)$.  Note that
  $s_2+s_3\neq 0$, since $d(s)\leq 4$.  If $s_3\neq 0$, then $s_3=1$
  and $s_2=0$; thus, we have $s_0+2s_1+2s_4=2$ and so $(2, 0, 0, 1, 0)$,
  $(0, 1, 0, 1, 0)$ and $(0, 0, 0, 1, 1)$ are the only elements in
  $X(G, 6)$ with $s_3\neq 0$. Assuming $s_3=0$, and $s_2\neq 0$, we
  have two cases $s_2=1$ and $s_2=2$. From the g.c.d condition in
  \eqref{kac-co}, we get that $s_2=1$ and $s_0+2s_1+2s_4=3$.  Thus,
  the tuples $(1, 1, 1, 0, 0)$, $(1, 0, 1, 0,1)$ and $(3, 0, 1, 0, 0)$
  are the only elements in $X(G, 6)$ with $s_3=0$ and $s_2\neq
  0$. Hence, $s=t$.
\subsection{Type ${\rm E}_6$}
We consider the case where $G$ is of the type $E_6$.
The conditions in \eqref{kac-co} are
$$s_0+s_1+2s_2+2s_3+3s_4+2s_5+s_6=m,\ (s_0, s_1,\dots, s_6)=1,$$
where $m\in \{3,4,6\}$. 
\subsubsection{}
Assume that $m=3$. Let $t=(0,0,0,0,1,0,0)$ be an
element in $X(E_6, 3)$.  Let $s\in X(E_6, 3)$ be an element
such that $d(s)=d(E_6, 3)$. Now, if $s_4\neq 0$, then
$s=(0,0,0,0,1,0,0)$. 
Any element $s\in X(E_6,3)$ with $s_4=0$ is up-to an extended
Dynkin diagram automorphism given by 
the elements $(1, 0,0,0,0,1,0)$, $(1,0,1,0,0,0,0)$
and $(1,1,0,0,0,0,1)$.  Thus $s=t$. 
\subsubsection{}
Consider the case where $m=4$.  Let $s\in X(E_6, 4)$. If $s_4\neq 0$,
then $s_4=1$, $s_2=s_3=s_5=0$ and $s_0+s_1+s_6=1$. Thus, any
$s\in X(E_6, 4)$ with $s_4\neq 0$ is an $\Omega$ translate of
$(1,0,0,0,1,0,0)$; and the centralizer of $x(s)$ is of type
$2A_2+A_1$. If $s_4=0$, then we cannot have $s_2=s_3=s_5=0$ as
$d(s)\leq 14$.  Assume that $s\in X(E_6, 4)$ with $d(s)=d(E_6,4)$,
then at least one of $s_2, s_3$ or $s_5$ is non-zero. If any two of
$s_2, s_3$ or $s_5$ are non-zero, then the centralizer of $x(s)$
contains a copy of $A_3+2A_1$, and $d(s)\geq 16$. Up to an action of
$\Omega$, we may assume that $s_2=s_3=0$ and $s_5=1$, or
$s_2=s_5=0$ and $s_3=0$
The centralizer of the elements $x(s)$ where $s=(1,1,0,0,0,1,0)$ and
$s=(1,0,0,1,0,0,1)$  are of the same type: 
$A_3+A_1$.

The root system of centralizer of $C_h^3$, with respect to $T$,  
is the subset of roots of $G$ with height divisible by 
$4$. These roots are the union of the sets
$$\{\beta_1=\alpha_1+\alpha_3+\alpha_4+\alpha_5, 
\beta_2=\alpha_6+\alpha_5+\alpha_4+\alpha_2,
\beta_1+\beta_2, -\beta_1, -\beta_2, -\beta_1-\beta_2\},$$
$$\{\gamma_1=\alpha_1+\alpha_3+\alpha_4+\alpha_2,
\gamma_2=\alpha_6+\alpha_5+\alpha_4+\alpha_3,
-\gamma_1,-\gamma_2-\gamma_1-\gamma_2\}$$ and
$\{\delta=\alpha_3+\alpha_4+\alpha_5+\alpha_2, -\delta\}$. These three
sets of roots are the connected components in the root system of type
$2A_2+A_1$.
\subsubsection{}
Consider the case $m=6$. Note that  $t=(1,1,0,0,1,0,1)$ belongs to
$X(E_6, 6)$ and the
 centralizer of $x(t)$ is of the type $3A_1$. 
 An element $s\in X(E_6, 6)$ such that
 $d(s)=d(E_6, 6)$ and $s\neq t$
 must have at least $5$ non-zero coordinates
and this is not possible.  
\subsection{Type ${\rm E}_7$}
We consider the case where $G$ is of type $E_7$ and here $m\in \{3,6,9\}$. 
Note that the conditions in \eqref{kac-co}
becomes,
$$s_0+s_7+2s_1+2s_2+2s_6+3s_3+3s_5+4s_4=m,\ 
(s_0, s_1,\dots, s_7)=1.$$
Moreover, we note that 
$\Omega=\mathbb{Z}/2\mathbb{Z}$. 
\subsubsection{}
Let $s\in X(E_7, 3)$.  When $m=3$, we note that $s_4=0$. If
$s_3\neq 0$, then $s=(0,0,0,1,0,0,0,0)$ and the centralizer of $x(s)$
is of the type $A_2+A_5$. In the case where $s_3=s_4=s_5=0$, we have
$s$, considered up to $\Omega$ action, is either $(1,1, 0,0,0,0,0,0)$,
$(1,0,1,0,0,0,0,0)$ or $(1,0,0,0,0,0,1,0)$--the corresponding
centralizers of $x(s)$ are of the type $D_6$, $A_6$ and $D_5+A_1$
respectively. Thus the centralizer with the smallest dimension is
$A_2+A_5$.
\subsubsection{}
Consider the case where $m=6$. Note that $t=(1,0,0,0,1,0,0,1)$ belongs
to $X(E_7, 6)$ and centralizer of $x(t)$ is of the type
$2A_2+A_1$. Let $s$ be a tuple in $X(E_7,6)$ such that
$d(s)=d(E_7, 6)$.  Assume that $s_4=0$. If
$|\{i: s_i\neq 0\}|>3$, then $s_0=s_7=1$ and $s_i=s_j=1$, for
some $i\neq j$ $i, j\in \{1,2,6\}$; and $d(s)\geq 20$. Thus, we get
that $|\{i:s_i=0\}|\geq 5$. If $s_3+s_5=0$, then $d(s)\geq
16$. Without loss of generality we assume that $s_3=1$. If $s_5=1$,
then $d(s)\geq 18$. Which implies that $s_5=0$ and since
$d(s)\leq 14$, we get that $s_2\neq 0$ and $s_6\neq 0$,  which is
impossible.  We conclude that $s_4\neq 0$. In this case $s_3=s_5=0$.
This implies that
  $$s_0+s_7+2s_1+2s_2+2s_6=2.$$
  From the above equation we get that $s$ is an $\Omega$ translate of
  $t$.
  \subsubsection{}
  Consider the case where $m=9$. The tuple $t=(1,0,1,1,0,0,1,1)$
  corresponds to an order $9$ element with centralizer of type
  $A_1+A_2$.  Let $s\in X(E_7, 9)$ be such that $d(s)=d(E_7, 9)$.
  Assume that $s_4\neq 0$. Since $d(s)\leq 8$, we note that
  $s_4=1$. If $s_3+s_5=0$, then $s_0+s_7$ is odd. If $s_0+s_7\geq 3$,
  then $d(s)>8$. In the case where $s_0+s_7=1$, without loss of
  generality we may assume that $s_0=0$ and $s_7=1$. This gives us
  that $s_1+s_2+s_6=2$.  Since $d(s)\leq 8$, we get that
  $s=(0, 1,0,0,1,0,1,1)$, and $x(s)$ has centralizer of type $4A_1$
  and $d(s)=8$. Assume that $s_3+s_5\neq 0$.  With out loss of
  generality, we assume that $s_3=1$ and $s_5=0$. If
  $s_1+s_2+s_6\neq 0$, then $d(s)\geq 12$. If $s_1=s_2=s_6=0$, then
  $d(s)\geq 10$. This concludes the discussion in the case where
  $s_4\neq 0$.
  
  Now, consider the case where $s_4=0$. If $d(s)=8$, then the rank of
  the root system of the centralizer of $x(s)$ is less than $4$; it is
  equal to $4$ if and only if $x(s)$ has centralizer of type
  $4A_1$. The case of $4A_1$ is clearly not possible for the Dynkin
  diagram of type $E_7$ with $s_4=0$. Thus centralizer of $x(s)$ has
  semisimple rank less than or equal to $3$. Assume that $\alpha_4$ is
  not connected to any vertex in the Dynkin diagram of the centralizer
  of $x(s)$, then $d(s)\geq 10$. Thus $4$ is connected to at least one
  more vertex in the Dynkin diagram of the centralizer of $x(s)$. If
  $\alpha_4$ and $\alpha_2$ are both roots of the centralizer of
  $x(s)$, then the fact that $d(s)\leq 8$ implies that
  $s_0+s_7+2s_1+2s_6\geq 4$, which is a contradiction to the
  assumption that $s\in X(E_7, 9)$. Thus with out loss of generality
  we may assume that $\alpha_4$ and $\alpha_5$ are both contained in
  the root system of the centralizer of $x(s)$. This implies that
  $s_1=0$ and $s_0=s_2=s_3=s_6=s_7=1$. In this case $s=t$.

  Note that the root system of the centralizer of the element $C_h^2$,
  an element of order $9$, consists of roots of height $9$. Since the
  Coxeter number of the root system ${E}_7$ is $18$, the centralizer
  of $C_h^2$ is not of type $A_1+A_2$. Thus, the centralizer of 
  $C_h^2$ is of type $4A_1$.
   \subsection{Type ${\rm E}_8$}
   Consider the case where $G$ is of type $E_8$. In this case
   $m\in \{3,5,6,10,15\}$.  The relations \eqref{kac-co} in this case
   are
  $$s_0+2s_1+2s_8+3s_2+3s_7+4s_3+4s_6+5s_5+6s_4=m$$
  and $(s_0, s_1,\dots, s_8)=1$. 
  \subsubsection{}
  We first consider the case where $m=3$. In this case, the
  centralizers of $x(s)$, for $s\in X(E_8,3)$ are of the type: $A_8$,
  $A_2+E_6$, $D_7$ and $E_7$. There exists an unique $s_0\in X(E_8,3)$
  such that $d(s_0)=d(E_8, 3)$, and the centralizer of $x(s_0)$ is of
  the type $A_2+E_6$.
\subsubsection{}
Consider the case $m=5$. The tuple $t=(0,0,0,0,0,1,0,0,0)$ belongs to
$X(E_8, 5)$. The partitions of $5$ are $(1,1,1,1,1)$, $(2, 1,1,1)$,
$(3,1,1)$, $(4,1)$, $(5)$, $(2,2,1)$ and $(3,2)$. Thus the non-zero
entries of for Kac coordinates are $5$, $(4,1)$, $(2,2,1)$ and
$(3,2)$.  If non-zero entries of $s\in X(E_8, 5)$ are $(4,1)$, then
the centralizers of $x(s)$ are of the type $A_1+A_6$, $D_7$ and
$E_7$. If non-zero entries of $s\in X(E_8, 5)$ are $(2,2,1)$, $s$ is
uniquely determined and the centralizer of $x(s)$ is of the type
$D_6$. If non-zero entries of $s\in X(E_8, 5)$ are $(2,3)$, then the
centralizers of $x(s)$ are of the type $A_7$, $D_5+A_2$, $A_6+A_1$ and
$E_6+A_1$. Hence, $t=(0,0,0,0,0,1,0,0,0)$ with centralizer $2A_4$ is
the unique element in $X(E_8, 5)$ such that $d(t)=d(E_8, 5)$.
  \subsubsection{}
  Consider the case where $m=6$. Partitions of $6$ with $1$ occurring
  at most once are as follows: $(1,5)$, $(1,2,3)$. The partitions of
  $6$ whose parts are at least $2$ are $(2,2,2)$, $(2,4)$,
  $(3,3)$. It is clear that the tuple $s\in X(E_8, 6)$ with $s_0=1$
  and $s_5=5$ and $s_i=0$, for all $i\neq 0,5$, is the unique element
  such that $d(s)=d(E_8, 6)$ and in this case the centralizer is $A_3+A_4$. 

\subsubsection{}
We consider the case where $m=15$. The tuple $t=(1,1,0,0,1,0,1,0,1)$
belongs to $X(E_8, 15)$ and the centralizer of $x(t)$ 
is of the type $4A_1$. 
For $s\in X(E_8, 15)$ we have,
$$15=\sum_{i=0}^8a_i(1-s_i)\leq \sum_{i:s_i=0}a_i.$$
Thus, we get that 
$$|\{i:0\leq i\leq 8, s_i=0\}|\geq 3,$$ 
with equality if and only if 
$$\{i:0\leq i\leq 8, s_i=0\}=\{4,5,6\}\ \text{or}\ \{3,4,5\},$$
moreover, for each $s\in X(E_8, 15)$ as above we have $d(s)\geq 12$.  Thus,
for $s\in X(E_8, 15)$ with $d(s)=d(E_8, 15)$, the semisimple rank is
at least $4$ and hence, it is of the type $4A_1$.  Let $I$ be the set
$\{i:0\leq i\leq 8, s_i=0\}$. If $4,5\not\in I$, then we get that
$\sum_{i\in I}a_i\leq 14$. Assume that $s_4=0$.  Then the set $I$ can
be one of the following $\{1,4,6,8\}$ or $\{1,4,7,0\}$, and
$\sum_{i\notin I}a_is_i\geq 16$. This implies that $s_4\neq 0$ and
$s_5=0$. In this case, the Dynkin diagram of the set $\{i: s_i= 0\}$
is $4A_1$ gives the following possibilities: $\{5,2,1,7\}$,
$\{5,2,1,8\}$, $\{5,2,1,0\}$, $\{5, 2,3,7\}$, $\{5, 2,3,8\}$,
$\{5, 2,3,0\}$, $\{5, 1,7,0\}$ and $\{5,3,7,0\}$. The condition that
$\sum_{i\in I}a_i\geq 15$ is only possible in the case where
$I=\{2,3,5,7\}$ and hence there exists a unique element
$s\in X(E_8, 15)$ such that $d(s)=d(E_8, 15)$.
\subsubsection{}
Finally, we turn to the case where $m=10$. The tuple
$t=(1,0,0,0,1,0,0,1,0)$ belongs to $X(E_8, 10)$ and the centralizer of
$x(t)$ is of the type $2A_2+2A_1$.  Let $s\in X(E_8, 10)$ such that
$d(s)=d(E_8, 10)$.

We will first obtain a contradiction to $s_4+s_5=0$. Assume that this
is indeed the case. Note that $s_3+s_6\neq 0$. If $s_3, s_6$ are both
non-zero then $s_3=s_6=1$ and there are three cases corresponding to
$s_0=2$, $s_1=2$ and $s_8=1$. The centralizers in each case are of
type $A_3+A_2+A_1$, $A_3+A_3$ and $3A_1+A_3$, clearly $d(s)>16$.
Thus, we conclude that exactly one among $s_3$ or $s_6$ is zero. Since
a copy of $A_4$ cannot be contained in the centralizer of $x(s)$, if
$s_6=0$, then $s_3, s_2, s_7$ are all non-zero. This is a
contradiction since $4s_3+3s_2+3s_7\geq 11$.  If $s_3=0$, then $s_6$,
$s_2$ and $s_1$ are non-zero. We infer that $s_0=1$. This implies that
the centralizer of $x(s)$ contains a copy of $A_3+A_2$.  Thus we get a
contradiction to the assumption that $s_4=s_5=0$.

  Assume that $s_5\neq 0$ and $s_4=0$. In this case, we have 
  $$|\{i\in \{2,7,3,6\}, s_i=0\}|\geq 3.$$ Then we can see that 
  there is a copy of $A_3+A_2$ in the centralizer and thus a
  contradiction to the assumption.
  
  Now consider the case where $s_5=0$ and $s_4\neq 0$. In this case,
  we have $s_3+s_6=0$, since if either $s_3$ or $s_6$ is non-zero then
  the centralizer of $x(s)$ contains $A_4+A_3$ or $A_4+A_2+A_1$. Note
  that $s_2+s_7$ cannot be zero, otherwise the centralizer of $x(s)$
  strictly contains $A_3+2A_1$. If $s_7=0$, then $A_4$ is contained in
  the centralizer of $x(s)$ and this is not possible. Thus $s_7\neq 0$
  and $s_2=0$, and in this case $s_1=s_8=0$ and $s_0=1$. Thus
  $s=(1, 0, 0, 0, 1, 0, 0, 1, 0)$ is the only tuple in $X(E_8, 10)$
  such that $d(s)=d(E_8, 10)$ and the centralizer of $x(s)$ of the
  type $2A_2+2A_1$.

      \end{proof}

      \vspace{4mm}

      {\bf Acknowledgment:} The authors owe to M. Reeder the
      observation that $C_m=C_h^{h/m}$, with $C_h$ the Coxeter
      conjugacy class in $\G(\C)$, may not be the only conjugacy class
      containing an element of order $m$ with minimal possible
      dimension of the centralizer group. He found that for $D_5,E_6$,
      there is one more besides $C_4$ containing an element of order
      4. This saved us the trouble of proving a wrong theorem! The
      papers \cite{Re1} and \cite{Re2} have a different focus, but
      contain results which in some cases are similar to what we do
      here. The authors thank P. Polo for discussions, for his
      reading of the earlier versions of  the paper and helpful comments,
      and for
      writing the paper \cite{Polo} completing the analysis of the
      group $\G(m)$ began here. The authors are grateful to J-P.Serre
      for his detailed
      comments on an earlier version of this paper that we shared with him,
and especially for the many (precise!) references that he provided.

\end{document}